\documentclass[10pt,twoside]{siamart1116}

\usepackage[english]{babel}
\usepackage{graphicx,epstopdf,epsfig}
\usepackage{amsfonts,epsfig,fancyhdr,graphics, hyperref,amsmath,amssymb}

\newcommand{\HE}{Name of Handling Editor}
\newcommand{\DoS}{May/15/2023}
\newcommand{\DoA}{Month/Day/Year}
\newcommand{\CA}{Daniel \v{S}ev\v{c}ovi\v{c}}

\newcommand{\Names}{So\v{n}a Pavl\'{\i}kov\'a, and Daniel \v{S}ev\v{c}ovi\v{c}}
\newcommand{\Title}{Qualitative, statistical and extreme properties of spectral indices of signable pseudo-invertible graphs}

\newtheorem{remark}[theorem]{Remark}

\renewcommand{\theequation}{\thesection.\arabic{equation}}
\renewtheorem{theorem}{Theorem}[section]

\usepackage{bm}
\usepackage{mathrsfs}

\begin{document}

\bibliographystyle{plain}

%  Leave these commented lines here
%\input{ELAheader-template.tex}
% ELA insert correct page number
\setcounter{page}{1}

\thispagestyle{empty}

%Insert the title of the paper
 \title{\Title\thanks{Received by the editors on \DoS.
 Accepted for publication on \DoA. 
 Handling Editor: \HE. Corresponding Author: \CA}}

\author{So\v{n}a Pavl\'{\i}kov\'a
\thanks{Inst. of Information Engineering, Automation, and Math, FCFT, Slovak Technical University, 812 37 Bratislava, Slovakia.  Supported by the grant APVV-22-0005.
}
\and
Daniel \v{S}ev\v{c}ovi\v{c}
\thanks{Department of Applied Mathematics and Statistics, FMFI, Comenius University, 842 48 Bratislava, Slovakia (sevcovic@fmph.uniba.sk). Supported by the grant APVV-20-0311. }
}

\markboth{\Names}{\Title}

\maketitle

\begin{abstract}
In this paper, we investigate the Moore-Penrose inversion  of a simple connected graph. We analyze qualitative, statistical, and extreme properties of spectral indices of signable pseudo-invertible graphs. We introduce and analyze a wide class of signable pseudo-invertible simple connected graphs. It is a generalization of the classical concept of positively integrally invertible graphs due to Godsil. We present several constructions of signable pseudo-invertible graphs. We also discuss statistical properties of various spectral indices of the class of signable pseudo-invertible graphs. 
\end{abstract}


% Sample article for the Electronic Journal of Linear Algebra

%%%%%%%%%%%%%%%%%%%%%%%%%%%%%%%%%%%%%%%

%%%%%%%%%%%%%%%%%%%%%%%%%%%%%%%%%%%%%%%%%%%%%%%%%%%%%%%%%%%%%

\section{Introduction}

In this paper, we investigate a class of signable pseudo-invertible graphs. A graph is called signable pseudo-invertible if the   Moore-Penrose inverse  matrix of its adjacency matrix can be signed to a matrix containing elements of the same sign. Recall that inversion of an adjacency matrix does not need to define a graph again because it may contain both positive and negative elements (see Harari and Minc \cite{Har}). To overcome this difficulty, Godsil \cite{Godsil1985} introduced the concept of positively integrally invertible graphs by defining a graph to be invertible if the inverse of its non-singular adjacency matrix is integral and diagonally similar (cf. Zaslavski \cite{Zas}) to a non-negative integral matrix representing the adjacency matrix of the inverse graph in which positive labels determine edge multiplicities. In the series of papers \cite{Pavlikova2016,Pavlikova2022-LAA},  Pavl\'\i kov\'a and \v{S}ev\v{c}ovi\v{c} extended this notion to a broader class of graphs by introducing the concept of negative invertibility of a graph. Both positively and negatively invertible graphs have the appealing property that inverting an inverse graph gives back the original graph. Related results, including a unifying approach to inverting graphs, were proposed in a recent survey paper by McLeman and McNicholas \cite{McMc} focusing on the inverses of bipartite graphs and the diagonal similarity to non-negative matrices. For other results on graph inverses based on Godsil's ideas, we refer to Akbari and Kirkland \cite{KirklandAkb2007}, Kirkland and Tifenbach \cite{KirklandTif2009}, and Bapat and Ghorbani \cite{Bapat} (see also \cite{Bapat2010}). Ye \emph{et al.} \cite{Ye} investigated graph inverses in the context of median eigenvalues. In \cite{Tifenbach},  Tifenbach investigated a class of graphs whose adjacency matrices are non-singular with integral inverses and strongly self-dual graphs. Pavl\'{\i}kov\'a \cite{Pavlikova2015} developed constructive methods to generate invertible graphs by the method of edge overlapping.  Our main purpose is to extend and investigate the concept of positive and negative invertibility of a graph due to Godsil \cite{Godsil1985}, Pavl\'\i kov\'a and \v{S}ev\v{c}ovi\v{c} \cite {Pavlikova2016} to the case when its adjacency matrix is not invertible but its   Moore-Penrose inverse  matrix is still signable to a matrix containing elements of the same sign. Such a signed   Moore-Penrose inverse  matrix can represent a weighted graph with non-negative weights. 

In applications including chemistry, biology, or statistics, various spectral indices of graphs representing the structure of organic molecules, or transition diagrams for finite Markov chains play an important role (cf. Cvetkovi\'c \cite{Cvetkovic2004, CvDS}, Brouwer and Haemers \cite{Brouwer2012} and references therein). All of them are related to various graph energies and spectral indices, which are characterized by means of eigenvalues of the adjacency matrix of a graph. Recall that the sum of absolute values of eigenvalues is called the matching energy index $\Lambda^{pow}$ (cf. Chen and Jinfeng \cite{Lin2016}), the maximum of absolute values of the least positive and largest negative eigenvalue is known as the HOMO-LUMO index $\Lambda^{ind}$ (see Mohar  \cite{Mohar2015}, Jakli\'c  \emph{et al.} \cite{Jaklic2012}, Fowler \emph{et al.} \cite{Fowler2010}), their difference is the HOMO-LUMO separation gap $\Lambda^{gap}$ (cf.  Gutman and Rouvray \cite{Gutman1979}). Regarding Aihara \cite{Aihara1999JCP}, a larger HOMO-LUMO gap implies higher kinetic stability and lower chemical reactivity of a molecule. According to Bacalis and Zdetsis \cite{Bacalis2009} the spectral separation gap $\Lambda^{gap}$ generally decreases with the number of vertices of the structural graph. The spectral indices $\Lambda^{gap}$ and $\Lambda^{ind}$ are closely related to the Moore-Penrose   (group)   inverse matrix $A^\dagger$ of the adjacency matrix $A$ of a given graph $G^A$. Following Pavl\'\i kov\'a and \v{S}ev\v{c}ovi\v{c} \cite{Pavlikova2016, Pavlikova2022-LAA} we have
\[
\Lambda^{gap}(A)=  \lambda_{max}(A^\dagger)^{-1} -  \lambda_{min}(A^\dagger)^{-1}, \quad
\Lambda^{ind}(A)=  \max( \lambda_{max}(A^\dagger)^{-1},  -\lambda_{min}(A^\dagger)^{-1} ),  
\]
where $\lambda_{max}(A^\dagger),  \lambda_{min}(A^\dagger)$ are the maximal and minimal eigenvalues of $A^\dagger$. As a consequence, properties of maximal and minimal eigenvalues of the pseudo-inverse graph can be used to analyze the spectral indices $\Lambda^{gap}$ and $\Lambda^{ind}$ of the graphs.

Recently, McDonald, Raju, and Sivakumar \cite{McDonald} studied the Moore-Penrose   (group)   inverses of adjacency matrices associated with certain graph classes. They derived formulae for the Moore-Penrose inverses of matrices that are associated with a class of digraphs obtained from stars. This new class contains both bipartite and non-bipartite graphs. A representation of the   Moore-Penrose inverse   matrix corresponding to the Dutch windmill graph has been derived by McDonald \emph{et al.} \cite{McDonald}. In \cite{Pavlikova2019-Australasian} Pavl\'{\i}kov\'{a} and \v{S}ir\'a\v{n} constructed a pseudo-inverse of a weighted tree in terms of maximal matchings and alternating paths. 
 
The maximum and minimal eigenvalues of the pseudo-inverse graph can be used to determine the spectral indices $\Lambda^{gap}$ and $\Lambda^{ind}$ (see also Pavl\'{\i}kov\'{a} and \v{S}ev\v{c}ovi\v{c} \cite{Pavlikova2015, Pavlikova2016, Pavlikova2022-LAA}).

The objective of this paper is two-fold. In Section 2 we introduce and investigate the wide class of signable pseudo-invertible simple connected graphs. Section 3 is devoted to computational and statistical results on this class of graphs. More precisely, in Section 2 we introduce the notion of positive/negative/positive and negative pseudo-invertible graphs, which is based on the signability of the Moore-Penrose inverse matrix $A^\dagger$ of the adjacency matrix $A$ corresponding to a graph $G^A$. The signed Moore-Penrose inverse matrix $A^\dagger$ defines a weighted graph $(G^A)^\dagger$. We investigate special classes of cycle and path graphs. We completely characterize which cycle and path graphs admit signable pseudo-inversion. We show that the complete graph $K_m$  is signable pseudo-invertible only for order $m=2$. Furthermore, we show that the complete multipartitioned graphs $K_{m_1,\dots,m_k}$ are signable pseudo-invertible only for $k=2$. Furthermore, we introduce and analyze a novel concept of $G^\mathscr{A}$-complete multipartitioned graph which is constructed from the original labeled graph $G^\mathscr{A}$ by replacing a vertex $i$ with the set of $m_i$ vertices and connecting them with the set of $m_j$ vertices iff $\mathscr{A}_{ij}=1$. It can be viewed as a natural generalization of a complete multipartitioned graph.  In Section 3 we discuss statistical properties of the class of signable pseudo-invertible graphs. We focus on maximal and minimal eigenvalues, as well as spectral indices $\Lambda^{gap}, \Lambda^{ind}$, and $\Lambda^{pow}$. In general, we present results for the order $m\le 10$ although some properties of extreme eigenvalues and indices hold for a general order $m$. Our descriptive statistical results are based on the complete list of simple connected graphs provided by McKay \cite{McKay} (see also \cite{Bender1990}) for orders $m\le 10$ in combination of a survey of signable pseudo-invertible graphs due to Pavl\'\i kov\'a and \v{S}ev\v{c}ovi\v{c} \cite{Pavlikova2022-WWW}.

\section{Signable pseudo-invertible graphs}

Let $G = (V, E)$ be an undirected connected graph with the set of $m$ vertices $V$ and the set of edges $E$. The graph $G$ may contain loops and its edges can be weighted. 
  
By $A_G$ we denote its symmetric adjacency matrix that contains   non-negative   elements. 
On the contrary, if $A$ is a nonnegative symmetric matrix, then $G^A$ denotes the graph with the adjacency matrix $A$. 

The spectrum $\sigma(G^A)$ of a graph $G^A$ consists of eigenvalues $\lambda_{min}\equiv\lambda_m \le \dots \le \lambda_1\equiv \lambda_{max}$ of its symmetric adjacency matrix $A_G$, that is $\sigma(G^A) = \sigma(A_G )$, (cf. Cvetkovi\'c \emph{et al.} \cite{Cvetkovic1978, CvDS}). Notice that the diagonal of an adjacency matrix $A$ representing a simple graph is zero, and so its trace is zero. As a consequence, we have $\lambda_{min}(A)<0<\lambda_{max}(A)$ where $\lambda_{min}(A), \lambda_{max}(A)$ are minimal and maximal eigenvalues of $A$, respectively. 

If $K$ is an $n\times m$ real matrix, then its Moore-Penrose inverse matrix is an $m\times n$ matrix $K^\dagger$ that is uniquely determined by the following identities (see Ben-Israel and Greville \cite{Ben}):
\begin{equation}
(K K^{\dagger})^T = K K^{\dagger}, \quad (K^{\dagger} K)^T = K^{\dagger} K, \quad K K^{\dagger} K =K, \quad K^{\dagger} K K^{\dagger}= K^{\dagger} .
\label{MPmn}
\end{equation}
In the case when $A$ is an $m\times m$ real symmetric matrix, Moore-Penrose can be constructed explicitly in the following way: Let $\mathscr{P}$ be an orthogonal matrix such that $\mathscr{P} A\mathscr{P}^{T}=\Lambda=\operatorname{diag}\left(\lambda_{1}, \ldots, \lambda_{m}\right)$, where $\lambda_{i}, i=1,\dots,m$, are real eigenvalues of the symmetric square matrix $A$. 
 
Given a number $\lambda\in\mathbb{R}$, its Moore-Penrose inverse is: $\lambda^\dagger = 0$ when $\lambda = 0$ and $\lambda^\dagger = 1/\lambda$, otherwise. 
 
Then the Moore-Penrose inverse matrix $A^\dagger$ that satisfies the axioms (\ref{MPmn}) is given uniquely by $A^{\dagger} = \mathscr{P}^{T}\Lambda^{\dagger}\mathscr{P}$ where $  \Lambda^{\dagger} =diag\left(\lambda_1^\dagger, \ldots, \lambda_{m}^\dagger\right)$. 
 
The Moore-Penrose inverse $A^\dagger$ of a real symmetric square matrix $A$ is also known as the group inverse of $A$. Note that the concept of a Moore-Penrose matrix inversion is more general than a matrix inversion. If a square symmetric matrix $A$ is invertible, then it is also   Moore-Penrose   invertible and $A^\dagger=A^{-1}$.

\begin{definition}\label{def:pseudoinv-matrix}
Let $A$ be a real symmetric matrix. Its Moore-Penrose inverse matrix $A^{\dagger}$ is called positively (negatively) signable if there exists a diagonal $\pm1$ signature matrix $D$ such that the matrix $D A^{\dagger} D$ contains non-negative (nonpositive) elements only. 
\end{definition}

\begin{definition}\label{def:pseudoinv}
An undirected weighted graph $G^A$ is called positively (negatively) pseudo-invertible if the Moore-Penrose inverse matrix $A^{\dagger}$ is a positively (negatively) signable matrix.  If $D$ is the corresponding signature matrix, the   undirected   pseudo-inverse graph $(G^A)^{\dagger}$ is defined by the weighted non-negative   symmetric   adjacency matrix $D A^{\dagger} D$, if $A^\dagger$ is positively signable ($ -DA^{\dagger} D$, if $A^{\dagger}$ is negatively signable).    A graph $G^A$ is called signably pseudo-invertible if it is positively or negatively pseudo-invertible. 
\end{definition}

In general, the symmetric Moore-Penrose inverse matrix $A^\dagger$ may contain non-integral elements (including its diagonal elements). The pseudo-inverse graph $(G^A)^\dagger$ is a uniquely defined undirected nonsimple graph containing weighted loops even if the graph $G^A$ is a simple graph.
 
Note that the pseudo-inverse graph $(G^A)^\dagger$ is defined by the non-negative weighted adjacency matrix $D A^\dagger D$ ($-D A^\dagger D$). The matrix $(D A^\dagger D)^\dagger$ is signable by the same signature matrix and $D (D A^\dagger D)^\dagger D = D D (A^\dagger)^\dagger D D = A$. One can proceed similarly if $G^A$ is a negatively pseudo-invertible graph. As a consequence, we obtain
\[
(G^\dagger)^\dagger = G.
\] 
It means that inverting an inverse graph gives the original graph. Furthermore, it is easy to verify by a simple contradiction argument that the weighted pseudo-inverse graph is connected, provided that the original graph is connected. 
 
Examples of simple connected graphs on five vertices and their weighted pseudo-invertible graphs with loops are shown in Fig.~\ref{fig-graf-m5-examples}.

\begin{figure}[ht]
    \centering
    
    \includegraphics[width=.25\textwidth]{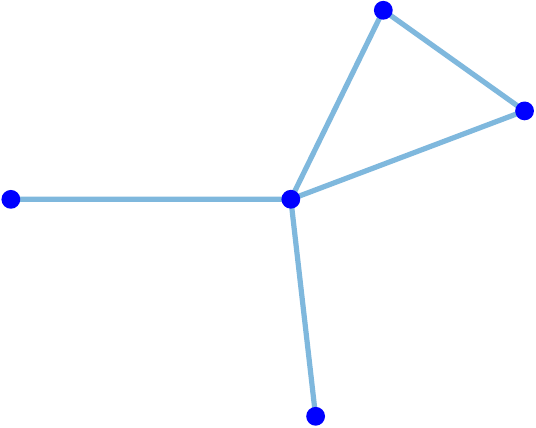} 
    \quad 
    \includegraphics[width=.25\textwidth]{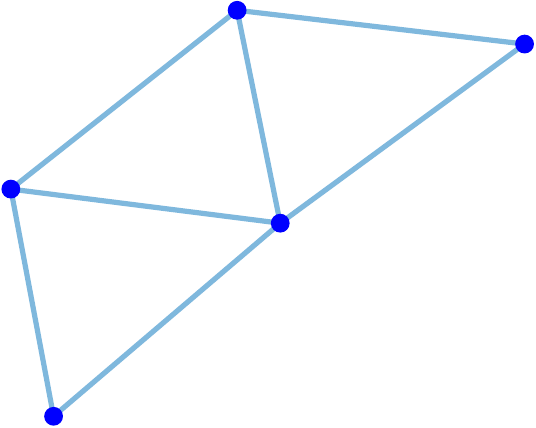}
    \quad 
    \includegraphics[width=.25\textwidth]{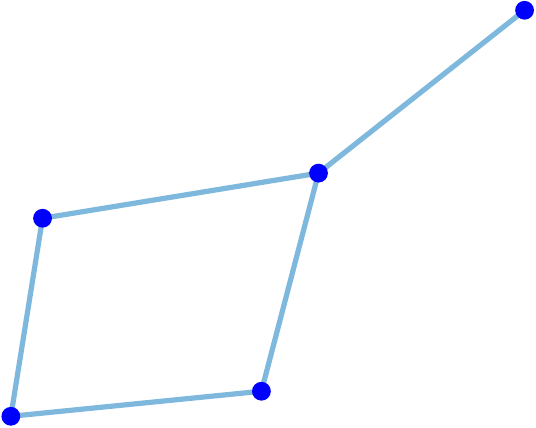}
\vskip 2mm

    \includegraphics[width=.25\textwidth]{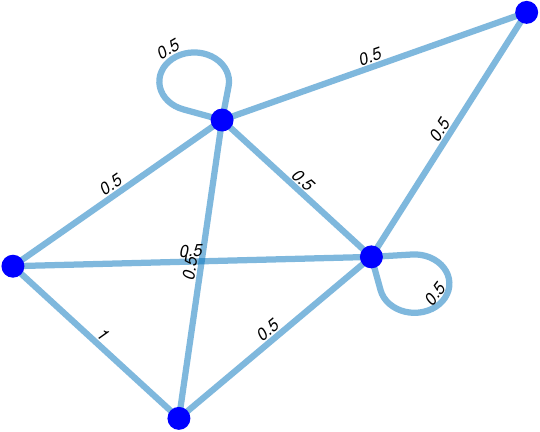} 
    \quad 
    \includegraphics[width=.25\textwidth]{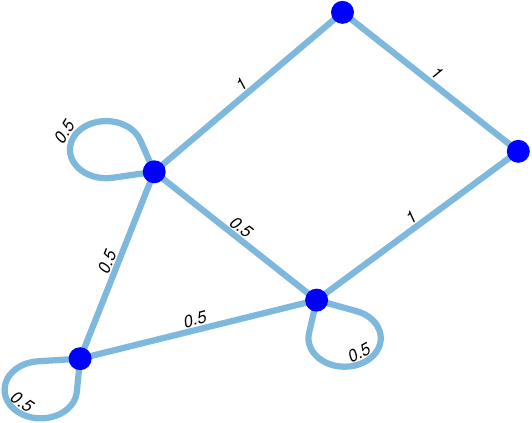}
    \quad 
    \includegraphics[width=.25\textwidth]{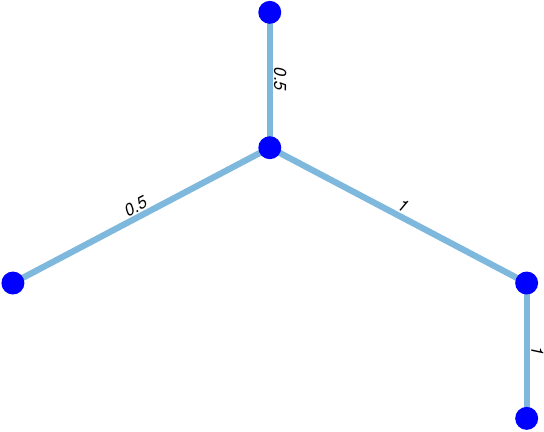}

    \caption{\small Top row: examples of a positively (left), negatively (middle), positively and negatively (right) pseudo-invertible graphs on $m=5$ vertices. Bottom row: corresponding weighted pseudo-inverse graphs.}
    \label{fig-graf-m5-examples}
    
\end{figure}

The explicit form of the Moore-Penrose (group) inverse of tridiagonal circulant matrices has been derived, e.g., by Encinas, Carmona \emph{et al.} \cite{carmona, encinas} and references therein. To our knowledge, the signability of   the Moore-Penrose   inverse matrix of a circulant or tridiagonal matrix has not been investigated so far.  
In the next proposition, we analyze positive/negative pseudo-invertibility of   cycle   $C_m$ and path $P_m$ graphs on $m$ vertices. 

\begin{proposition}
Let $C_m$ be a   cycle   graph of order $m$, $m\ge 3$. The graph $C_m$ is neither positively nor negatively pseudonvertible for $m\not=4$. The   bipartite   graph $C_4$ is positively and negatively pseudo-invertible.
Path graphs $P_m$ are integrally positively and negatively invertible bipartite graphs for any even order $m$. The path graph $P_3$ is positively and negatively pseudo-invertible. The path graph $P_m$ is neither positively nor negatively pseudonvertible for any odd degree $m\not=3$. 

\end{proposition}

\begin{proof}
There are several papers dealing with the form of the   Moore-Penrose   inverse matrix to the adjacency matrix of a circular graph $C_m$ (see Sivakumar and Nandi \cite{Sivakumar2022}, and references therein). Here, we present the approach taking into account the circulant form of the adjacency matrix. Notice that the adjacency matrix $A$ has the form of a circular Toeplitz matrix, $A_{ij}=a_{i-j}, i,j=1,\dots,m$, where $a_p, p\in\mathbb{Z}$, is an $m$-periodic sequence such that $a_{m-p}=a_{-p}=a_p, a_1=a_{-1}=a_{m-1}=1$, and $a_p=0$ for $p\not=\pm1, m-1, -m+1$. The    Moore-Penrose   inverse matrix $A^\dagger$ can be searched again in the form of a circular Toeplitz matrix, $A^\dagger_{ij}=a^\dagger_{i-j}, i,j=1,\dots,m$. Taking into account the form of the sequence $a_p, p\in\mathbb{Z}$, it is straightforward to verify that $A^\dagger$ satisfies the Moore-Penrose axioms provided that
$\sum_{p=1}^m (a_{s-p-1} + a_{s-p+1} ) a^\dagger_p  = a_s, \quad\text{for each}\  s=1,\dots, m$. 
Calculating the solution of this equation and taking into account the fact $a^\dagger_p=a^\dagger_{-p}=a^\dagger_{m-p} $ we obtain:
\begin{itemize}
    \item if $m \equiv 0$\ (mod $4$), then $a^\dagger_p = \frac{1}{m} \left( \frac{m}{2} - |p| \right) (-1)^{\frac{|p|-1}{2}}$ for $p$ odd,  $a^\dagger_p =0$ for $p$ even; 
    
    \item if $m \not\equiv 0$\ (mod $4$), then $a^\dagger_p = \frac{1}{2} (-1)^{\frac{|p|-1}{2}}$ for $p$ odd. For $p$ even we have:
    
 $a^\dagger_{p} = \frac{1}{2} (-1)^{\frac{|p|}{2}}$, if  $m \equiv 1$\ (mod $4$);
\ \  $a^\dagger_p =0$, if  $m \equiv 2$\ (mod $4$);
 \ \  $a^\dagger_{p} = \frac{1}{2} (-1)^{\frac{|p|}{2}+1}$, if  $m \equiv 3$\ (mod $4$).
    
\end{itemize}

Note that $a^\dagger_1=1/2$, and $a^\dagger_3=-1/2$ for any order $m\not=4$. Assume that the cycle graph $C_m$ is positively pseudo-invertible. Then there exists a signature matrix $D=diag(d_1, \dots, d_m)$ signing $A^\dagger$ to a non-negative matrix. Without loss of generality, we may assume $d_1=1$. Then $d_{i+1} A^\dagger_{i+1, i} d_i = d_{i+1} a^\dagger_{1} d_i \ge 0$. Since $a^\dagger_1>0$, then $d_2>0$, and subsequently $d_p=1$, for $p=2,\dots,m$, that is, $D=I$. As the matrix $A^\dagger$ contains both positive and negative elements for $m\not= 4$, we conclude that $A^\dagger$ is not positively signable for $m\not=4$. Similarly, assuming $A^\dagger$ is negatively signable matrix, i.e. $D A^\dagger D\le 0$, by the signature matrix $D$ leads to the conclusion $D=diag(1,-1,1, \dots , (-1)^{m+1})$. For $m\ge 5$ we have $d_1 A^\dagger_{14} d_4= d_1 a^\dagger_3 d_4 =(-1/2)(-1)>0$, a contradiction. For the triangle $C_3$ of order $m=3$ we have $d_3 A^\dagger_{31} d_1 = A^\dagger_{31} =a^\dagger_2 =1/2$,  a contradiction. 
For $m=4$ the matrix $A^\dagger$ is proportional to the adjacency matrix $A$ of the cycle graph $C_4$. It is positively signable by the diagonal matrix $D=diag(1,1,1,1)$, and negatively signable by the signature matrix $D=diag(1,-1,1,-1)$. Hence $C_4$ is the only signable cycle graph. It is positively and negatively signable.

The path graph $P_m$ has the tridiagonal adjacency matrix $A$ where $A_{ij}=1$ for $|i-j|=1$, and $A_{ij}=0$, otherwise. It follows from the general results on the inverses of tridiagonal Toeplitz matrices due to Da Fonseca and Petronilho \cite[Corollary 4.2, Section 3]{fonseca} that the Toeplitz matrix $A$ is invertible for any even order $m$, and $det(A)=(-1)^{\frac{m}{2}}$. Furthermore, for $i\le j$, and $m$ even we have
\[
(A^{-1})_{ij} = (A^{-1})_{ji} = \left\{
\begin{array}{rl}
0, & \text{if $i$ is even, or $j$ is odd,}  \\ 
1, & \text{if $i+j \equiv 1$\ (mod $4$), $i$ is odd, and $j$ is even,} \\
-1,& \text{if $i+j \equiv 3$\ (mod $4$), $i$ is odd, and $j$ is even.} 
\end{array}
\right. 
\]
For $i\le j$, and $m$ odd we have
\[
(A^{\dagger})_{ij} = (A^{\dagger})_{ji} = (-1)^{\frac{j-i-1}{2}}\left\{
\begin{array}{cl}
0, & \text{if $i+j$ is even,}  \\ 
\frac{i}{m+1}, & \text{if $i$ is even, $j$ is odd,} \  \\
\frac{m-j+1}{m+1}, & \text{if $i$ is odd, $j$ is even.} \  
\end{array}
\right. 
\]
If the order $m$ is even, then the matrix $A^{-1}$ is positively signable by the diagonal signature matrix $D=diag(d_1, \dots, d_m)$ where $d_i=1$ for $i\equiv 1$, or $i\equiv 2$ (mod $4$), and $d_i=1$ for $i\equiv 3$, or $i\equiv 4$ (mod $4$). It is also negatively signable $D A^{-1} D\le 0$ by the signature matrix $D$ where $d_i=1$ for $i\equiv 1$, or $i\equiv 4$ (mod $4$), and $d_i=-1$ for $i\equiv 2$, or $i\equiv 3$ (mod $4$). That is, the path graph $P_m$ is positively and negatively integrally invertible bipartite graph for any even order $m$.

For the order $m=3$ we have $A^\dagger = \frac{1}{2} A$. Therefore, the path $P_3$ is positively and negatively pseudo-invertible bipartite graph. On the other hand, if $m\ge 5$ is odd,  then $A^\dagger$ is neither positively nor negatively signable.   In fact, suppose that $A^\dagger$ is positively signable by the signature matrix $D$. As $(A^\dagger)_{i, i+1}>0, (A^\dagger)_{i, i+3}<0$ we have $d_1 d_2 >0, d_2 d_3 >0, d_3 d_4 >0$, but $d_1 d_4 <0$, a contradiction. A similar argument shows that $A^\dagger$ cannot be a negatively signable matrix. 
\end{proof}

For the order $k=2$ the complete graph $K_2$ is just the simple path graph $P_2$ which is a positive and negative integrally invertible graph. It is the only example of a simple connected graph that is self-invertible $K_2^{-1} \cong K_2$ (cf. Harary and Minc \cite{Har}). On the other hand, the complete graph $K_k, k\ge 3$, is neither positive nor negatively pseudo-invertible.

\begin{proposition}
\label{Kkgraph}
The complete graph $K_k$ of order $k\ge 3$ has an invertible adjacency matrix $\mathscr{A}$, but it is neither positively nor negatively    Moore-Penrose   invertible. 
\end{proposition}

\begin{proof}
The adjacency matrix $\mathscr{A}$ of the complete graph $G^{\mathscr{A}}=K_k$ and its inverse matrix $\mathscr{A}^{-1}$ have the form:
\[
\mathscr{A}={\bf 1} {\bf 1}^T - I = \left(\alpha_{ij}\right)_{i,j=1,\dots,k}, \qquad \mathscr{A}^{-1}=\frac{1}{k-1}{\bf 1} {\bf 1}^T - I,
\]
where ${\bf 1}= (1,\dots, 1)^T\in\mathbb{R}^k$, and $\alpha_{ii}=0, \alpha_{ij}=1$ for $i\not=j$. Let $\mathscr{D}=diag(d_1, \dots, d_k)$ be a signature matrix. Then, inspecting the diagonal terms of $\mathscr{D}\mathscr{A}^{-1}\mathscr{D}$ we conclude $(\mathscr{D}\mathscr{A}^{-1} \mathscr{D})_{ii} = (1/(k-1) -1) (d_{ii})^2<0$. Therefore, $G^{\mathscr{A}}$ cannot be positively pseudo-invertible. On the other hand, the matrix $\mathscr{A}^{-1}$ contains positive off-diagonal and negative diagonal entries. Therefore, any signature matrix $\mathscr{D}$ that signs $\mathscr{A}^{-1}$ to a non-positive matrix should have two elements $d_{i'}$ and $d_{j'}$ of the same sign, $i'\not= j'$. Then $(\mathscr{D} \mathscr{A}^{-1} \mathscr{D})_{i'j'} = d_{i'}  (\mathscr{A}^{-1})_{i'j'} d_{j'} >0$, a contradiction. Therefore, $G^{\mathscr{A}}=K_k$ is neither positively nor negatively pseudo-invertible graph, as claimed.
\end{proof}

\begin{figure}[ht]
    \centering

    \includegraphics[width=.25\textwidth]{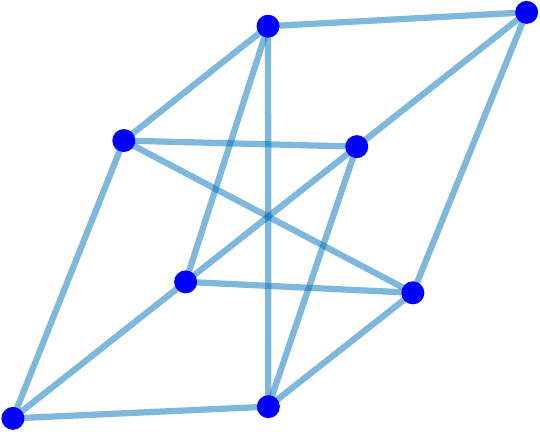}
    \qquad\qquad 
    \includegraphics[width=.25\textwidth]{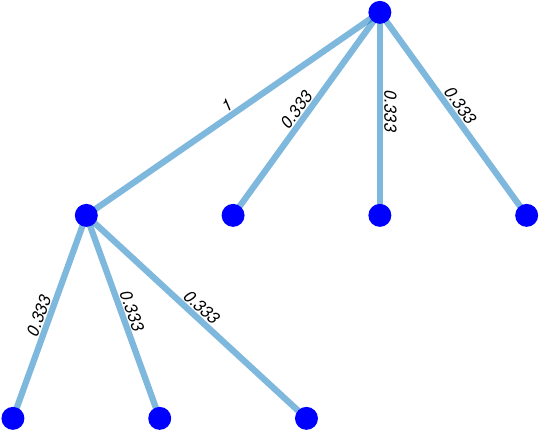}

    \caption{\small The uncomplete positively and negatively psudoinvertible bipartite graph $K_{m,m}^{-e}$ with $m=4$ (left) and its pseudo-inverse weighted graph $(K_{m,m}^{-e})^\dagger$ (right).
    }
    \label{fig-graf-Kmm-e}
\end{figure}

In the next proposition, we show that the complete bipartite graph $K_{m_1,m_2}$ of the order $m=m_1 + m_2$ 
is positively and negatively pseudo-invertible. Notice that $K_{m-1,1}=S_m$ is just the star graph and $K_{2,2}=C_4$ is the cycle graph.

\begin{proposition}
The complete multipartitioned graph $K_{m_1,m_2}$ of the order $m=m_1 + m_2$ is positively and negatively pseudo-invertible. The pseudo-inverse graph $(K_{m_1,m_2})^\dagger$ is the weighted bipartite graph $K_{m_1,m_2}$ with all weights equal to $1/(m_1 m_2)$.
\end{proposition}

\begin{proof}
Let $\mathscr{E}_{m_1, m_2}$ be an $m_1\times m_2$ matrix consisting of ones. It is easy to verify that $(\mathscr{E}_{m_1, m_2})^\dagger = (1/(m_1 m_2)) \mathscr{E}_{m_2, m_1}$. Therefore, the adjacency matrix $A$  of the graph $K_{m_1,m_2}$, and its   Moore-Penrose   inverse matrix $A^\dagger$ have the form
\[
A = \left(
\begin{array}{cc}
0 &  \mathscr{E}_{m_1, m_2}\\
(\mathscr{E}_{m_1, m_2})^T & 0
\end{array}
\right), \quad
A^\dagger  = \frac{1}{m_1 m_2} 
\left(
\begin{array}{cc}
0 &  \mathscr{E}_{m_1, m_2}\\
(\mathscr{E}_{m_1, m_2})^T & 0
\end{array}
\right) = \frac{1}{m_1 m_2} A .
\]
The matrix $A^\dagger$ is positively (negatively) signable by the signature matrix $D=diag(I,I)$ ($D=diag(I, -I)$). The proof of the proposition follows. 
\end{proof}

Let us denote by $K_{m,m}^{-e}$ the noncomplete bipartite graph constructed from the complete bipartite graph $K_{m,m}$ by deleting exactly one edge. In \cite{Pavlikova2022-DM},  Pavl\'{i}kov\'a, \v{S}ev\v{c}ovi\v{c}, and \v{S}ir\'a\v{n} showed that its spectrum consists of $2m -4$ zeros and four real eigenvalues $\lambda^{\pm,\pm}=\pm\left( 1-m \pm \sqrt{m^2 + 2m -3}\right)/2$. In the following proposition we will prove the positive and negative pseudo-invertibility of $K_{m,m}^{-e}$, and we completely characterize the pseudo-inverse weighted graph $(K_{m,m}^{-e})^\dagger$. 

\begin{proposition}
    \label{Kmm-e}
The bipartite noncomplete graph $K_{m,m}^{-e}$ is positively and negatively pseudo-invertible. The pseudo-inverse graph $(K_{m,m}^{-e})^\dagger$ is the weighted graph consisting of two star graphs $S_m$ having edge weights equal to $1/(m-1)$ and connected through the central vertices by an edge with unit weight. 
\end{proposition}

\begin{proof}
Without loss of generality, we may assume that the adjacency matrix $A$  of the graph $K_{m,m}^{-e}$ has the form
\[
A = \left(
\begin{array}{cc}
0 & K\\
K^T & 0
\end{array}
\right), \qquad K = \mathbf{1}\mathbf{1}^T - e_1 e_1^T,
\]
where $\mathbf{1}=(1,\dots,1)^T, e_1=(1,0,\dots, 0)^T\in\mathbb{R}^m$.
It is straightforward to verify that the Moore-Penrose inverse $A^\dagger$ has the form
\begin{equation}
A^\dagger = \left(
\begin{array}{cc}
0 & (K^\dagger)^T\\
K^\dagger & 0
\end{array}
\right), \qquad K^\dagger = \frac{1}{m-1} \left( e_1 \mathbf{1}^T  +  \mathbf{1} e_1^T - (m+1) e_1 e_1^T \right).
\label{bipartite-matrix}
\end{equation}
Clearly, if $D_+ = diag(-1, 1, \dots, 1)$, and $D_- = diag(1, -1, \dots, -1)$, then the matrix $D_+ K^\dagger D_-$ contains non-negative elements only. Therefore, $D A^\dagger D\ge 0$ where $D=diag(D_+, D_-)$. Thus, the graph $G^A=K_{m,m}^{-e}$ is positively pseudo-invertible. If we take the signature matrix $D=diag(D_+, - D_-)$, then $D A^\dagger D\le 0$. As a consequence, the graph $K_{m,m}^{-e}$ is also negatively pseudo-invertible. 

Finally, $(K^\dagger)_{11}= -1$, and $(K^\dagger)_{1j}= (K^\dagger)_{i1} = 1/(m-1)$ for $i,j\not=1$. Therefore, the pseudo-inverse graph $(K_{m,m}^{-e})^\dagger$ is the weighted graph consisting of two star graphs $S_m$ having edge weights equal to $1/(m-1)$ and connected through the central vertices by an edge with unit weight (see Fig.~\ref{fig-graf-Kmm-e}). 
\end{proof}

The spectrum of a complete multipartioned graph $K_{m_1, \dots, m_k}$ has been investigated by Delorme \cite{Del} (see also S. Pavl\'{\i}kov\'{a}, D. \v{S}ev\v{c}ovi\v{c}, J. \v{S}ir\'a\v{n} \cite{Pavlikova2022-DM}). In what follows, we propose a novel concept of $G^\mathscr{A}$-complete multipartitioned graph denoted by $G^\mathscr{A}_{m_1, \dots, m_k}$. It is a natural generalization of a complete multipartitioned graph $K_{m_1, \dots, m_k}$. If $G^\mathscr{A} = K_k$ is the complete graph, then $G^\mathscr{A}_{m_1, \dots, m_k}$ is just the complete multipartioned graph $K_{m_1, \dots, m_k}$.

\begin{definition}
    \label{GAmultipartitioned}
    Let $G^\mathscr{A}$ be a simple connected vertex labeled graph of the order $k$ with an adjacency matrix $\mathscr{A}$, and vertex labels $\{1,\dots, k\}$. Assume $m_1, \dots, m_k\ge 1$. We denote by $G^\mathscr{A}_{m_1, \dots, m_k}$ the $G^\mathscr{A}$-complete multipartitioned vertex labelled graph of the order $m=m_1 + \dots + m_k$, constructed as follows: Each vertex $i$ of $G^\mathscr{A}$ is replaced by the set $V_i$ of $m_i$ vertices. Each vertex of $V_i$ is connected to every vertex of $V_j$ provided there is an edge connecting the vertex $i$ and $j$ in the original graph $G^\mathscr{A}$, that is, $\mathscr{A}_{ij}=1$. 
\end{definition}

\begin{proposition}
\label{GAmultipartitioned-theorem}
Let $G^\mathscr{A}$ be a simple connected vertex labelled graph of the order $k$ with an adjacency matrix $\mathscr{A}$, and vertex labels $\{1,\dots, k\}$. Assume $m_1, \dots, m_k\ge 1$. Let $G^\mathscr{A}_{m_1, \dots, m_k}$ be the $G^\mathscr{A}$-complete multipartitioned graph of order $m=m_1 + \dots + m_k$. Then
\begin{itemize}
    \item[i)] its spectrum $\sigma(G^\mathscr{A}_{m_1, \dots, m_k})$ consists of the zero eigenvalue with multiplicity $m-k$, and all eigenvalues of the $k\times k$ matrix $\mathscr{M}^{1/2} \mathscr{A} \mathscr{M}^{1/2}$ where $\mathscr{M} =diag(m_1, \dots, m_k)$. 
    
    \item[ii)]  The   Moore-Penrose   inverse matrix $(\mathscr{M}^{1/2} \mathscr{A} \mathscr{M}^{1/2})^\dagger$ is positively (negatively) signable iff the graph $G^\mathscr{A}_{m_1, \dots, m_k}$ is positively (negatively) pseudo-invertible. 
    
    \item[iii)]  Suppose that the adjacency matrix $\mathscr{A}$ is invertible. Then the graph $G^\mathscr{A}$ is positively (negatively) invertible iff the graph $G^\mathscr{A}_{m_1, \dots, m_k}$ is positively (negatively) pseudo-invertible. 
\end{itemize}
\end{proposition}

\begin{proof}
The adjacency matrix $A$ of $G^\mathscr{A}_{m_1, \dots, m_k}$ has the block form $A= \left(\alpha_{ij}  \mathscr{E}_{m_i,m_j} \right)_{i,j=1,\dots,k}$ where the elements $\alpha_{ij}$ form the adjacency matrix $\mathscr{A} = \left(\alpha_{ij}\right)_{i,j=1,\dots,k}$ of the underlying graph $G^\mathscr{A}$, and $\mathscr{E}_{m_i,m_j}$ is the $m_i\times m_j$ matrix consisting of ones. 

To prove i), the vector $x=(x^{(1)}, \dots, x^{(k)}) \in \mathbb{R}^{m_1} \times \dots \times \mathbb{R}^{m_k} \equiv \mathbb{R}^{m}$, where $x^{(i)} \in \mathbb{R}^{m_i}$, is an eigenvector of the adjacency matrix $A$, that is, $A x = \lambda x$, iff $\sum_{j=1}^k \alpha_{ij} \xi_j  = \lambda x^{(j)}_p$ for each $p=1, \dots, m_i$, and $i=1,\dots, k$, where $\xi_j = \sum_{p=1}^{m_j} x^{(j)}_p = m_j x^{(j)}_1$, that is, $x^{(i)}_1=x^{(i)}_2 = \dots x^{(i)}_{m_j}$. As $x\not=0$, we have $\xi\not=0$.  Therefore, the non-trivial vector $\xi\in\mathbb{R}^k$ is a solution to the linear system of equations $\sum_{j=1}^k \alpha_{ij} m_j \xi_j  = \lambda \xi_i$ for $i=1,\dots, k$. That is $\lambda\in \sigma( \mathscr{A} \mathscr{M}) = \sigma(\mathscr{M}^{1/2} \mathscr{A} \mathscr{M}^{1/2})$, as claimed in part i). 

We search the    Moore-Penrose   inverse matrix $A^\dagger$ in block matrix form $A^\dagger = \left(\beta_{ij}  \mathscr{E}_{m_i,m_j} \right)_{i,j=1,\dots,k}$ where the elements $\beta_{ij}$ form the symmetric matrix $\mathscr{B} = \left(\beta_{ij}\right)_{i,j=1,\dots,k}$. Since $\mathscr{E}_{m_i m_p} \mathscr{E}_{m_p m_j} =m_p \mathscr{E}_{m_i,m_j}$ we have 
\[
A A^\dagger = \left( (\mathscr{A} \mathscr{M} \mathscr{B})_{ij}  \mathscr{E}_{m_i,m_j} \right)_{i,j=1,\dots,k},
\qquad 
A^\dagger A = \left( (\mathscr{B} \mathscr{M} \mathscr{A})_{ij}  \mathscr{E}_{m_i,m_j} \right)_{i,j=1,\dots,k} , 
\]
where $\mathscr{M} =diag(m_1, \dots, m_k)$. Taking $\mathscr{B} = \mathscr{M}^{-1/2} (\mathscr{M}^{1/2} \mathscr{A} \mathscr{M}^{1/2})^{\dagger} \mathscr{M}^{-1/2}$ we obtain 
\[
A A^\dagger  = (A A^\dagger )^T =
\left( 
( \mathscr{M}^{-1/2} (\mathscr{M}^{1/2} \mathscr{A} \mathscr{M}^{1/2}) (\mathscr{M}^{1/2} \mathscr{A} \mathscr{M}^{1/2})^{\dagger} \mathscr{M}^{-1/2} )_{ij}  \mathscr{E}_{m_i,m_j} \right)_{i,j=1,\dots,k}.
\]
Now, it is easy to verify $A A^\dagger A =A, \ A^\dagger A A^\dagger = A^\dagger$, and so $A^\dagger$ is indeed the Moore-Penrose inverse of the adjacency matrix $A$. 

In order to prove statement ii), let us assume that the graph $G^\mathscr{A}_{m_1, \dots, m_k}$ is a positively (negatively) pseudo-invertible graph. Then there exists a $m\times m$ diagonal signature matrix:
\begin{equation}
D = diag(d^{(1)}_1,\dots, d^{(1)}_{m_1}, \ \dots\ ,\  \dots,  d^{(k)}_1,\dots, d^{(k)}_{m_k})
\label{Dmat}
\end{equation}
such that $D A^\dagger D \ge 0$ ($\le 0$). Let $\mathscr{D} = diag( d^{(1)}_1,\dots , d^{(k)}_1)$ be the $k\times k$ signature diagonal matrix. Since $A^\dagger = \left(\beta_{ij}  \mathscr{E}_{m_i,m_j} \right)_{i,j=1,\dots,k}$, then $\mathscr{D} \mathscr{B} \mathscr{D} \ge 0$ ($\le 0$) where $\mathscr{B} = \left(\beta_{ij}\right)_{i,j=1,\dots,k} = \mathscr{M}^{-1/2} (\mathscr{M}^{1/2} \mathscr{A} \mathscr{M}^{1/2})^{\dagger} \mathscr{M}^{-1/2}$. As $\mathscr{D} \mathscr{M}^{-1/2} =  \mathscr{M}^{-1/2} \mathscr{D}$, and $\mathscr{M}^{-1/2}>0$ the matrix $\mathscr{D} (\mathscr{M}^{1/2} \mathscr{A} \mathscr{M}^{1/2})^{\dagger} \mathscr{D}$ has the same sign of elements as the matrix $\mathscr{B}$ we conclude that   Moore-Penrose   inverse matrix $(\mathscr{M}^{1/2} \mathscr{A} \mathscr{M}^{1/2})^\dagger$ is positively (negatively) signable.
On the contrary, if the    Moore-Penrose   inverse matrix $(\mathscr{M}^{1/2} \mathscr{A} \mathscr{M}^{1/2})^\dagger$ is positively (negatively) signable, then the matrix $\mathscr{A}^\dagger$ is signable to a non-negative (nonpositive) matrix by a signature matrix $\mathscr{D} = diag( d^{1}, \dots , d^{k})$. Then the matrix $A^\dagger$ is signable to a positive (negative) matrix by the signature matrix of the form (\ref{Dmat}) with $d^{(i)}_j = d^i$ for $j=1,\dots, m_i, i =1,\dots, k$, and the proof follows. 

To prove iii), we suppose that the matrix $\mathscr{A}$ is invertible. Since the diagonal matrices $\mathscr{D}$ and $\mathscr{M}^{1/2}$ commute, we have $\mathscr{D}(\mathscr{M}^{1/2} \mathscr{A} \mathscr{M}^{1/2})^{\dagger} \mathscr{D} = \mathscr{D}(\mathscr{M}^{1/2} \mathscr{A} \mathscr{M}^{1/2})^{-1} \mathscr{D} =  \mathscr{M}^{-1/2} \mathscr{D} \mathscr{A}^{-1} \mathscr{D} \mathscr{M}^{-1/2}$. It means that the matrix $(\mathscr{M}^{1/2} \mathscr{A} \mathscr{M}^{1/2})^\dagger$ is positively (negatively) signable if and only if the graph $G^\mathscr{A}$ is positively (negatively) pseudo-invertible. The rest of the proof of iii) now follows from part ii).
\end{proof}

\begin{proposition}
The complete multipartitioned graph $K_{m_1, \dots, m_k}$ of the order $m=m_1 + \dots + m_k$, where $k\ge 3$ is neither positively nor negatively pseudo-invertible. If $k<m$, then the adjacency matrix $A$ is singular. 
\end{proposition}

\begin{proof}
    
With regard to Proposition~\ref{Kkgraph} the complete graph $K_k, k\ge 3$, is neither positively nor negatively pseudo-invertible.  Its adjacency matrix $\mathscr{A}$ is invertible. The rest of the proof now follows from Proposition~\ref{GAmultipartitioned-theorem}, part iii).

\end{proof}

\begin{figure}[ht]
    \centering
    
    \includegraphics[width=.25\textwidth]{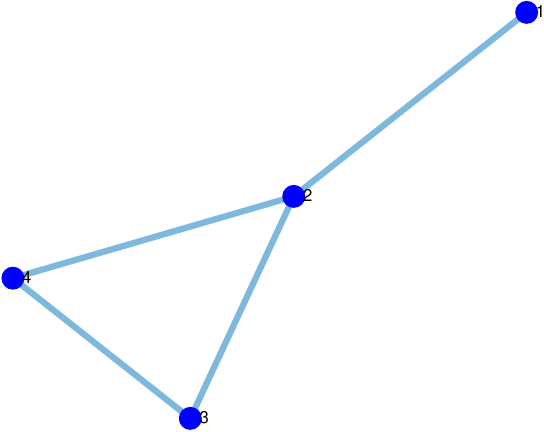} 
    \quad 
    \includegraphics[width=.25\textwidth]{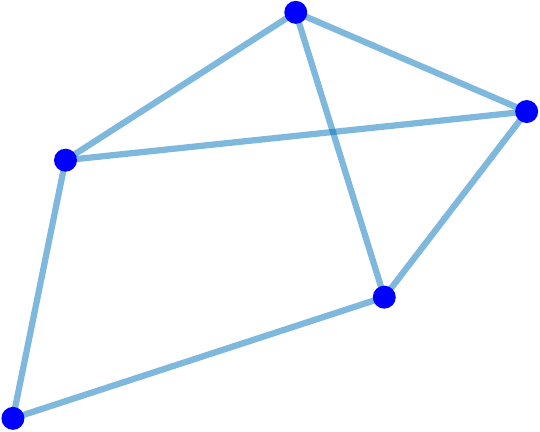}
    \quad 
    \includegraphics[width=.25\textwidth]{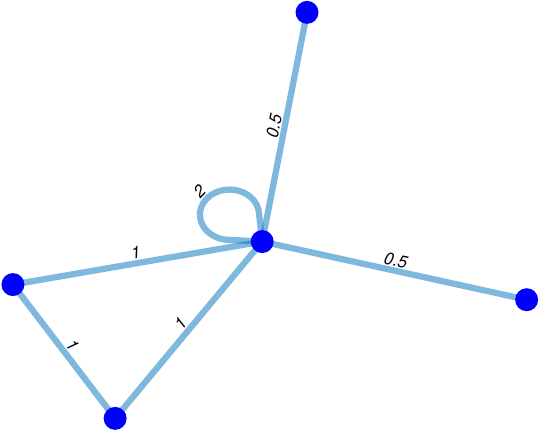}
    
    \caption{\small A positively integrally invertible vertex labeled graph $G^{\mathscr{A}}$ with  $k=4$ vertices (left); the $G^\mathscr{A}$ complete graph $G^{\mathscr{A}}_{1,2,1,1}$ complete graph (middle); the weighted signable pseudo-inverse graph $(G^{\mathscr{A}}_{1,2,1,1})^\dagger$ (right).}
    \label{fig-graf-Ascr-A-pinvA}
    
\end{figure}

\begin{definition}
    \label{GAselfsimilar}
    Let $G^A$ be a simple connected graph of order $m$ with an adjacency matrix $A$. We say that the graph $G^A$ is homothetically self-pseudoinvertible iff $G^A$ is a signable pseudo-invertible graph, and $(G^A)^\dagger\sim G^A$. i.e., there exists $\kappa\in\mathbb{R}$ and a permutation matrix $P$ of order $m$ such that $D A^\dagger D = \kappa P A P^T$ where $D$ is the signature matrix such that $D A^\dagger D$ (or $-D A^\dagger D$) is the adjacency matrix of $G^\dagger$.
\end{definition}

\begin{remark}
Path graphs $P_2$, $P_3$, $P_4$ and the cycle graph $C_4$ are homothetically self-pseudoinvertible graphs. Furthermore, complete bipartite graphs $K_{m_1, m_2}, m=m_1+m_2, m\ge 2$, form an infinite family of homothetically self-pseudoinvertible graphs. In Proposition~\ref{prop-ii-pendant} we show that a graph $G^A$ constructed from a given graph $G^B$ by adding a pendant vertex to each vertex of $G^B$ is always a signable integrally invertible graph whose inverse graph $(G^A)^{-1}$ is homothetically similar to $G^A$. 

\end{remark}

\section{Statistical and extreme properties of eigenvalues and spectral indices of signable  pseudo-invertible and signable integrally invertible graphs}

Let $G^A$ be a simple connected graph with an adjacency matrix $A$. Then $A$ has positive and negative eigenvalues, because $trace(A)=0$. In what follows, we shall denote $\lambda_\pm(G^A)\equiv \lambda_\pm(A)$, the least positive and largest negative eigenvalues of the adjacency matrix $A$.  Let us denote by $\Lambda^{gap}(A)= \lambda_+(A) -  \lambda_-(A)$ and $\Lambda^{ind}(A)= \max(|\lambda_+(A)|, |\lambda_-(A)|)$ the spectral gap and the spectral index of $A$. Furthermore, we define the spectral power $\Lambda^{pow}(A) = \sum_{k=1}^m |\lambda_k|$. Clearly, all three spectral indices $\Lambda^{gap}, \Lambda^{ind}$, and $\Lambda^{pow}$ depend on the positive $\sigma_+(A)=\{\lambda \in \sigma(A), \lambda>0\}$, and negative $\sigma_-(A)=\{\lambda \in \sigma(A), \lambda<0\}$ parts of the spectrum of the matrix $A$. In fact, $\lambda_+(A) = \min\sigma_+(A), \lambda_-(A) = \max\sigma_-(A)$, and $\Lambda^{pow} = \sum_{\lambda\in\sigma_+(A)} \lambda - \sum_{\lambda\in\sigma_-(A)} \lambda = 2 \sum_{\lambda\in\sigma_+(A)} \lambda$. In the context of spectral graph theory, the eigenvalues of the adjacency matrix $A$ representing a structural chemical graph of an organic molecule play an important role. The spectral gap $\Lambda^{gap}(A)$ is also known as the HOMO-LUMO energy separation gap of the energy of the highest occupied molecular orbital (HOMO) and the lowest unoccupied molecular orbital orbital (LUMO). Generally speaking, the molecule is more stable when the spectral gap is larger (cf. Aihara \cite{Aihara1999JCP}).

\begin{figure}[ht]
    \centering
    
    \includegraphics[width=.3\textwidth]{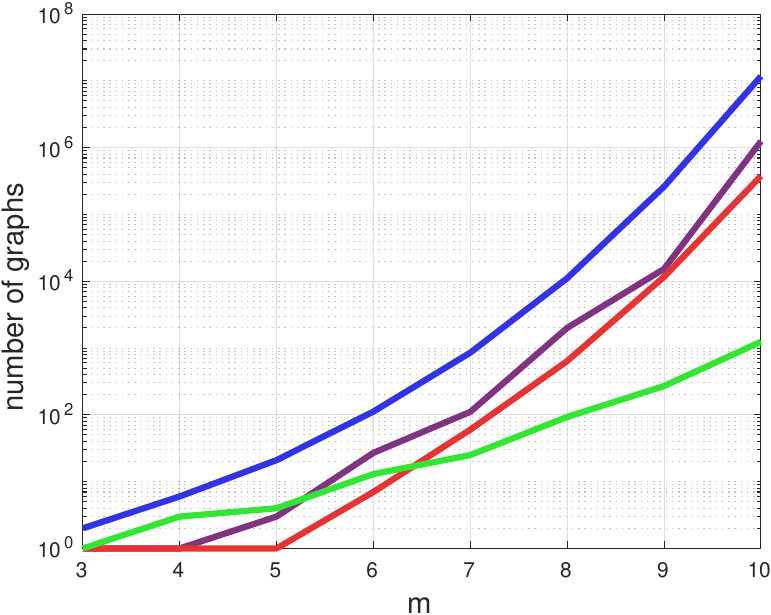} 
    \quad\quad 
    \includegraphics[width=.3\textwidth]{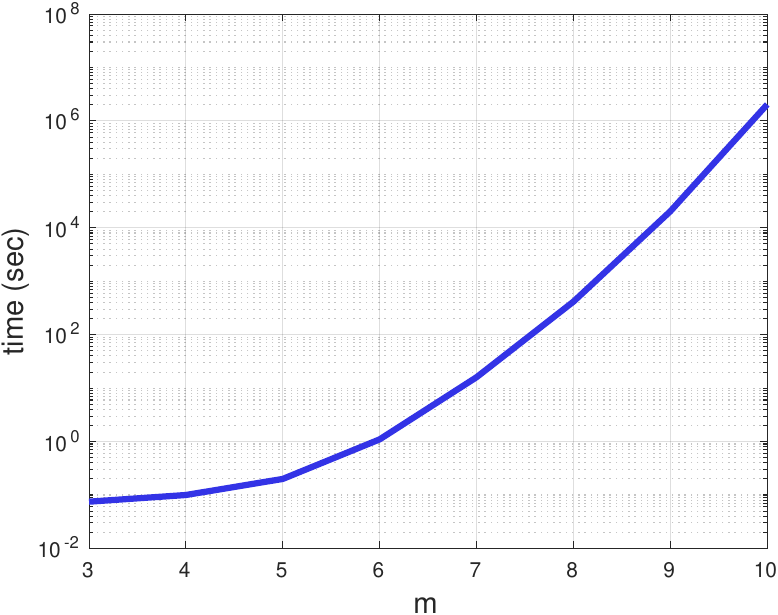}
    
    \caption{\small Left: the number of all simple connected graphs of order $m\le 10$ (blue). The numbers of positively but not negatively pseudo-invertible graphs (magenta), the number of negatively but not positively pseudo-invertible graphs (red), the number of positively and negatively (bipartite) pseudo-invertible graphs (green). Right: computational time complexity of the results summarized in Table~\ref{tab-1}.}
    \label{fig-graf-complexity-approximate}
    
\end{figure}

If we denote by $c_m$ the number of simple nonisomorphic connected graphs on $m$ vertices, then the number $c_m$ can be approximated by the quadratic exponential function: $c_m\approx \omega_0 10^{\omega_1 (m-9)+\omega_2 (m-9)^2}$, 
where $\omega_0=261080, \omega_1=1.4, \omega_2=0.09$ (cf. \cite{Pavlikova2022-DM}). This formula is exact for the order $m=9$ and gives accurate approximation results for other $m\le10$ (see Fig.~\ref{fig-graf-complexity-approximate}, left (blue line)). Moreover, $\log(c_m) = O(m^2)$ as $m\to \infty$. The class of multipartitioned complete graphs is completely characterized by its partitions $\{ m_1, \dots, m_k\}$. The number $\hat{c}_m$ of multipartitioned complete graphs is therefore given by the number of partitions of order $m=m_1 + \dots + m_k$, for $k=1,\dots, m$. According to Hardy-Ramanujan's 1918 result, we have $\hat{c}_m \approx \frac{1}{4 m\sqrt{3}} \exp(\pi\sqrt{2m/3})$, i.e. $\log(\hat{c}_m) = O(\sqrt{m})$ as $m\to \infty$. This means that the number of all multipartitioned complete graphs of a given order $m$ is considerably smaller than the number of all simple connected graphs of that order. 

Recall the following well-known facts regarding the minimal and maximal eigenvalues of a graph. The maximal value of $\lambda_{max}=\lambda_1$ on all simple connected graphs on $m$ vertices is equal to $m-1$, and is reached by the complete graph $K_m$. The minimal value of $\lambda_{max}$ is equal to $2\cos(\pi/(m+1))$, and is achieved for the path graph $P_m$. The lower bound for the minimal eigenvalue $\lambda_{min}=\lambda_m \ge -\sqrt{\lfloor m/2\rfloor\lceil m/2\rceil}$ was independently proved by Constantine  \cite{Con} and Powers \cite{Pow}. The lower bound is attained for the complete bipartite graph $K_{m_1,m_2}$ where $m_1=\lceil m/2\rceil, m_2=\lfloor m/2\rfloor$. The maximal value of $\lambda_{min}$ on all simple connected graphs on the $m$ vertices is equal to $-1$. It is  attained for the complete graph $K_m$. Unfortunately, neither the complete graph $K_k$ nor the complete multipartitioned graph $K_{m_1, \dots, m_k}$ are signable pseudo-invertible graphs for $k\ge 3$. This is why we have to investigate the properties of extreme spectral indices within the class of signable pseudo-invertible graphs that do not include complete graphs $K_k$, or complete multipartitioned graphs $K_{m_1, \dots, m_k}$ with $k\ge 3$.

\begin{table}
\small
\caption{\small The number of all simple connected graphs  $G^A$ on $m\le 10$ vertices, graphs $G^A$ with invertible adjacency matrix ($det(A)\not=0$), and graphs $G^A$ with integrally invertible adjacency matrix ($det(A)=\pm1$). 
Source: own computations \cite{Pavlikova2022-WWW} based on McKay's list of all simple connected graphs \cite{McKay}.
}
 \label{tab-summary}
 \medskip
 
\centering 
\scriptsize
\begin{tabular}{l||r|r|r|r|r|r|r|r|r}
\hline
$m$              &$2$&$3$&$4$&$5$&$6$&$7$&$8$&$9$&$10$  \\
\hline
all graphs       & 1   & 2   & 6   & 21  & 112 & 853 & 11117 & 261080 &11716571 \\
$det(A)\not=0$   & 1   & 1   & 3   &  8  & 52  & 342 &  5724 & 141063 & 7860195 \\
$det(A)=\pm1$      & 1   & -   & 2   &  -  & 29  &   - &  2381 & -      & 1940904 \\
\hline

\end{tabular}

\end{table}

In Table~\ref{tab-summary} we present the number of all simple connected graphs $G^A$ on $m\le 10$ vertices, graphs $G^A$ with an invertible adjacency matrix ($det(A)\not=0$), and graphs $G^A$ with an integrally invertible adjacency matrix ($det(A)=\pm1$) based on our calculations \cite{Pavlikova2022-WWW}, and McKay's list of all simple connected graphs \cite{McKay}. In Table \ref{tab-1} we present the number of positively but not negatively pseudo-invertible graphs ($+$signable). The number of negatively  but not positively  pseudo-invertible graphs ($-$signable). The number of simultaneously positively and negatively invertible and pseudo-invertible graphs ($\pm$signable). Complete characterization of these classes of graphs and their spectrum can be found in \cite{Pavlikova2022-WWW}. 
In Fig.~\ref{fig-graf-complexity-approximate} we depict the dependence of the number of all simple connected graphs of the order $m\le 10$ (blue). The number of (+signable) graphs (magenta), the number (-signable) of graphs (red), and the number of ($\pm$signable) bipartite graphs (green). We also show the computational time complexity of the results summarized in Table~\ref{tab-1}.

Finally, in Table~\ref{tab-summary-ii} we present the number of positively but not negatively integrally invertible graphs ($+$signable). Then the number of negatively  but not positively  integrally invertible graphs ($-$signable). We also present the number of positively and negatively integrally invertible graphs ($\pm$signable).

\begin{table}
\small
\caption{\small The number of positively but not negatively pseudo-invertible graphs ($+$signable). The number of negatively  but not positively  pseudo-invertible graphs ($-$signable). The number of simultaneously positively and negatively invertible and pseudo-invertible graphs ($\pm$signable). Source: own computations \cite{Pavlikova2022-WWW}.
}

\label{tab-1}
\scriptsize

\begin{center}
\begin{tabular}{l||r|r|r|r|r|r|r|r|r}
\hline
        $m$         &$2$&$3$&$4$&$5$&$6$&$7$&$8$&$9$&$10$  \\
\hline
$+$signable          & 0   & 0   & 1   &  3  & 27  & 111 &  2001 & 15310  & 1247128 \\
%\hline
$-$signable          & 0   & 0   & 0   &  1  &  7  &  60 &   638 & 11643  & 376137  \\
%\hline
$\pm$signable        & 1   & 1   & 3   &  4  & 13  &  25 &    93 &   270  & 1243     \\
\hline
all signable             & 1   & 1   & 4   &  8  & 47  & 196 &  2732 & 27223  & 1624508 \\
\hline
\end{tabular}

\end{center}

\end{table}

\begin{table}
\small
\caption{\small The number of signable positively but not negatively integrally invertible graphs ($+$signable). The number of negatively  but not positively  integrally invertible graphs ($-$signable). The number of positively and negatively integrally invertible graphs ($\pm$signable). 
Source: own computations \cite{Pavlikova2022-WWW} based on McKay's list of all simple connected graphs \cite{McKay}.
}
 \label{tab-summary-ii}
 \medskip
 
\centering 
\scriptsize
\begin{tabular}{l||r|r|r|r|r|r|r|r|r}
\hline
$m$                 &$2$&$3$&$4$&$5$&$6$&$7$&$8$&$9$&$10$  \\
\hline
$+$signable          & 0   & -   & 1   &  -  & 20  &   - &  1626 & -      & 1073991 \\
%\hline
$-$signable          & 0   & -   & 0   &  -  &  4  &   - &  260  & -      & 105363  \\
%\hline
$\pm$signable        & 1   & -   & 1   &  -  &  4  &   - &  25   & -      & 349     \\
\hline
all signable         & 1   & 0   & 2   &  0  & 28  & 0   &  1911 & 0   & 1179703 \\
\hline

\end{tabular}

\end{table}

\subsection{All simple connected signable pseudo-invertible graphs}

\begin{table}[ht]
 \caption{\small Descriptive statistics of the maximal(minimal) eigenvalues $\lambda_{max}$ ($\lambda_{min}$), spectral gap $\Lambda^{gap}$, spectral index $\Lambda^{ind}$, and spectral power $\Lambda^{pow}$ for all signable pseudo-invertible simple connected graphs on $m\le 10$ vertices.
}
\scriptsize
\begin{center}
\hglue -0.5truecm \begin{tabular}{l||l|l|l|l|l|l|l|l}
\hline
$m$  &  $3$ & $4$ & $5$ & $6$ & $7$ & $8$   & $9$    & $10$\\
\hline\hline
$E(\lambda_{max})$ &
 1.4142 & 
 1.8800 & 
 2.4066 & 
 2.8107 & 
 3.3106 & 
 3.7471 & 
 4.2491 & 
 4.7634 
\\
$\sigma(\lambda_{max})$  &
 -- & 
 0.2510 & 
 0.3993 & 
 0.5074 & 
 0.5640 & 
 0.6160 & 
 0.6553 & 
 0.6309  
\\
${\mathcal S}(\lambda_{max})$ &
  -- & 
 0.1191 & 
 -0.0036 & 
 -0.0101 & 
 -0.1702 & 
 -0.0756 & 
 -0.0282 & 
 -0.1001 
\\
${\mathcal K}(\lambda_{max})$  &
  -- & 
 1.3963 & 
 1.6392 & 
 2.0501 & 
 2.2927 & 
 2.7045 & 
 2.9230 & 
 2.9675 
\\
$\max(\lambda_{max})$ &
 1.4142 $(\pm)$& 
 2.1701 $(+)$& 
 2.9354 $(-)$& 
 3.7321 $(+)$& 
 4.4253 $(+)$& 
 5.9164 $(-)$& 
 7.0315 $(-)$& 
 8.1231 $(-)$
\\
$\min(\lambda_{max})$ &
 1.4142 $(\pm)$& 
 1.6180 $(\pm)$& 
 1.8478 $(\pm)$& 
 1.8019 $(\pm)$& 
 1.9319 $(\pm)$& 
 1.8794 $(\pm)$& 
 1.9616 $(\pm)$& 
 1.9190 $(\pm)$
\\
\hline
$E(\lambda_{min})$  &
 -1.4142 & 
 -1.7078 & 
 -1.9739 & 
 -2.1375 & 
 -2.3811 & 
 -2.4676 & 
 -2.6930 & 
 -2.7947 
\\
$\sigma(\lambda_{min})$ &
 -- & 
 0.2201 & 
 0.2708 & 
 0.3243 & 
 0.3311 & 
 0.3198 & 
 0.3259 & 
 0.2838 
\\
${\mathcal S}(\lambda_{min})$  &
  -- & 
 -0.4543 & 
 -0.4438 & 
 -0.6230 & 
 -0.5321 & 
 -0.6962 & 
 -0.5011 & 
 -0.4281 
\\
${\mathcal K}(\lambda_{min})$  &
  -- & 
 1.8907 & 
 2.2189 & 
 2.9493 & 
 2.9185 & 
 3.8455 & 
 3.3743 & 
 3.4984 
\\
$\max(\lambda_{min})$ &
 -1.4142 $(\pm)$& 
 -1.4812 $(+)$& 
 -1.6180 $(-)$& 
 -1.6180 $(-)$& 
 -1.7823 $(-)$& 
 -1.6180 $(-)$& 
 -1.8384 $(+)$& 
 -1.6180 $(-)$
\\
$\min(\lambda_{min})$  &
 -1.4142 $(\pm)$& 
 -2.0000 $(\pm)$& 
 -2.4495 $(\pm)$& 
 -3.0000 $(\pm)$& 
 -3.4641 $(\pm)$& 
 -4.0000 $(\pm)$& 
 -4.4721 $(\pm)$& 
 -5.0000 $(\pm)$
\\
\hline
$\max(\Lambda^{gap})$ &
 2.8284 $(\pm)$& 
 4.0000 $(\pm)$& 
 4.8990 $(\pm)$& 
 6.0000 $(\pm)$& 
 6.9282 $(\pm)$& 
 8.0000 $(\pm)$& 
 8.9442 $(\pm)$& 
 10.000 $(\pm)$
\\
$\max(\Lambda^{ind})$ &
 1.4142 $(\pm)$& 
 2.0000 $(\pm)$& 
 2.4495 $(\pm)$& 
 3.0000 $(\pm)$& 
 3.4641 $(\pm)$& 
 4.0000 $(\pm)$& 
 4.4721 $(\pm)$& 
 5.0000 $(\pm)$
\\
$\max(\Lambda^{pow})$ &
 2.8284 $(\pm)$& 
 4.9624     $(+)$& 
 7.1068     $(-)$& 
 8.8284     $(+)$& 
 11.2176    $(-)$& 
 14.000     $(+)$& 
 16.7446    $(-)$& 
 19.4136    $(-)$
\\
\hline
$\min(\Lambda^{gap})$  &
 2.8284 $(\pm)$& 
 1.2360 $(\pm)$& 
 1.0806     $(-)$& 
 0.7423     $(+)$& 
 0.6429     $(-)$& 
 0.3877     $(+)$& 
 0.3310     $(-)$& 
 0.1647     $(+)$
\\
$\min(\Lambda^{ind})$&
 1.4142 $(\pm)$& 
 0.6180 $(\pm)$& 
 0.6180     $(-)$& 
 0.4142 $(\pm)$& 
 0.3573     $(-)$& 
 0.2624     $(+)$& 
 0.1937     $(-)$& 
 0.1092     $(+)$
\\
$\min(\Lambda^{pow})$ &
 2.8284 $(\pm)$& 
 3.4642 $(\pm)$& 
 4.0000 $(\pm)$& 
 4.4722 $(\pm)$& 
 4.8990 $(\pm)$& 
 5.2916 $(\pm)$& 
 5.6568 $(\pm)$& 
 6.0000 $(\pm)$
\\
\hline
\end{tabular}
\end{center}
\label{tab-spektrumpositivepseudo}
\end{table}

In Table~\ref{tab-spektrumpositivepseudo} we present descriptive statistics of  maximal(minimal) eigenvalues $\lambda_{max}$($\lambda_{min}$), spectral gap $\Lambda^{gap}$, spectral index $\Lambda^{ind}$, and spectral power $\Lambda^{pow}$ for all signable pseudo-invertible simple connected graphs on $m\le 10$ vertices.
The symbols $E,\sigma, {\mathcal S}$ and ${\mathcal K}$ represent the mean value, standard deviation, skewness, and kurtosis of the corresponding sets of eigenvalues $\lambda_{max}$, and $\lambda_{min}$, respectively. Skewness ${\mathcal S}(\lambda_{max})$ is close to zero, and kurtosis ${\mathcal K}(\lambda_{max})$ tends to $3$. This means that the distribution of maximal eigenvalues of all signable pseudo-invertible simple connected graphs on $m$ vertices becomes normally distributed as $m$ increases. On the other hand, ${\mathcal S}(\lambda_{min})<0$ and ${\mathcal K}(\lambda_{min})>3$. It reveals that the distribution of minimal eigenvalues is slightly skewed to the left, and it has a leptokurtic distribution with fat tails. The signatures $(+)/(-)/(\pm)$ after the extreme value indicate (positive)/(negative)/(positive and negative) pseudo-invertibility of the graph attaining this extreme value.

Recall that the lower bound for $\lambda_{min}$ is achieved for the complete bipartite graph $K_{m_1,m_2}$ where $m_1=\lceil m/2\rceil, m_2=\lfloor m/2\rfloor$. The graph $K_{m_1,m_2}$ is positively and negatively pseudo-invertible. In Fig.~\ref{fig-graf-lambdamax-signablepseudoinvertible-even} we show signable pseudo-invertible graphs on $3\le m \le10$ vertices with the maximal value of $\lambda_{max}$. For values of $\lambda_{max}$ we refer to Table~\ref{tab-spektrumpositivepseudo}.

\begin{figure}[ht]
    \centering

    \includegraphics[height=.15\textwidth, angle=90]{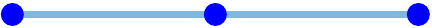}
    \ 
    \includegraphics[width=.2\textwidth]{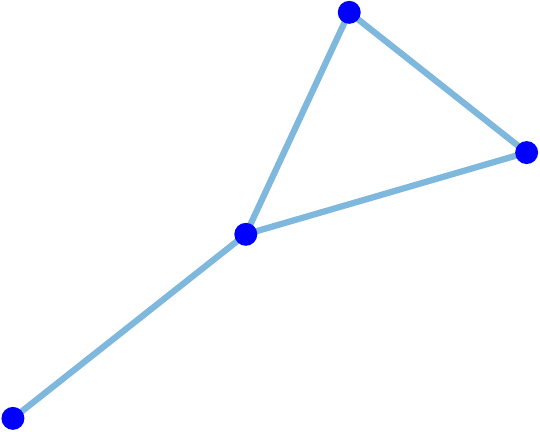}
    \ 
    \includegraphics[width=.2\textwidth]{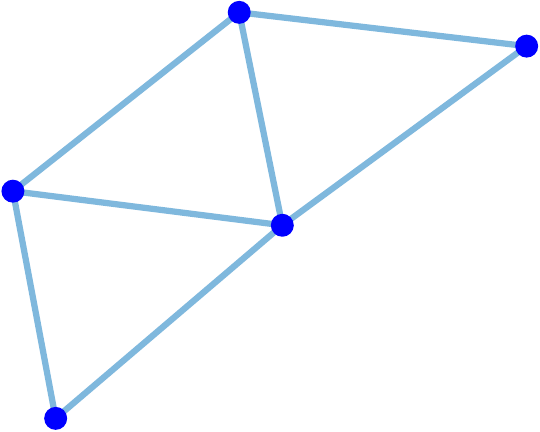}
    \ 
    \includegraphics[width=.2\textwidth]{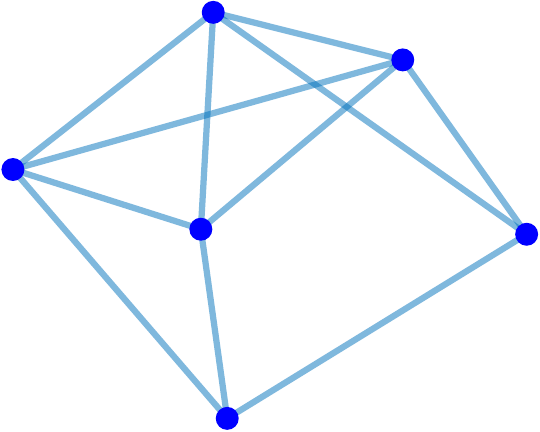}
    
$m=3$ \hskip 2.4truecm $m=4$ \hskip 2.6truecm $m=5$ \hskip 2.8truecm $m=6$
\vskip 0.5truecm 

    \includegraphics[width=.23\textwidth]{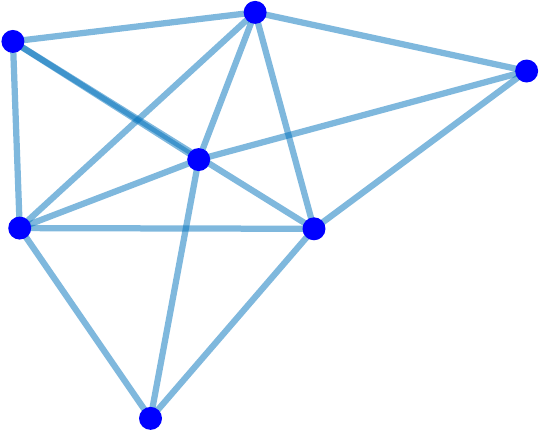}
    \ 
    \includegraphics[width=.23\textwidth]{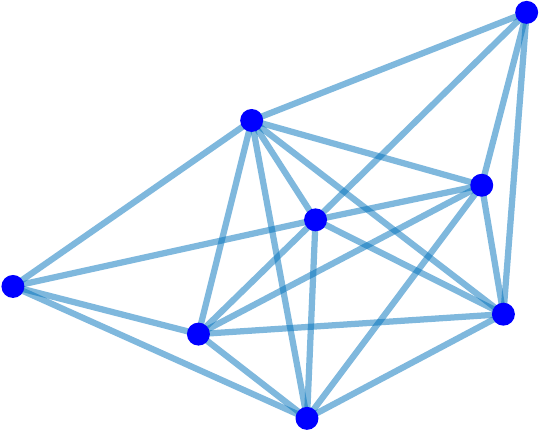}
    \ 
    \includegraphics[width=.23\textwidth]{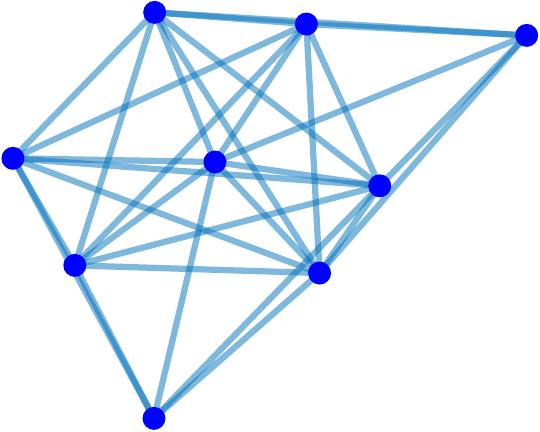}
    \
    \includegraphics[width=.24\textwidth]{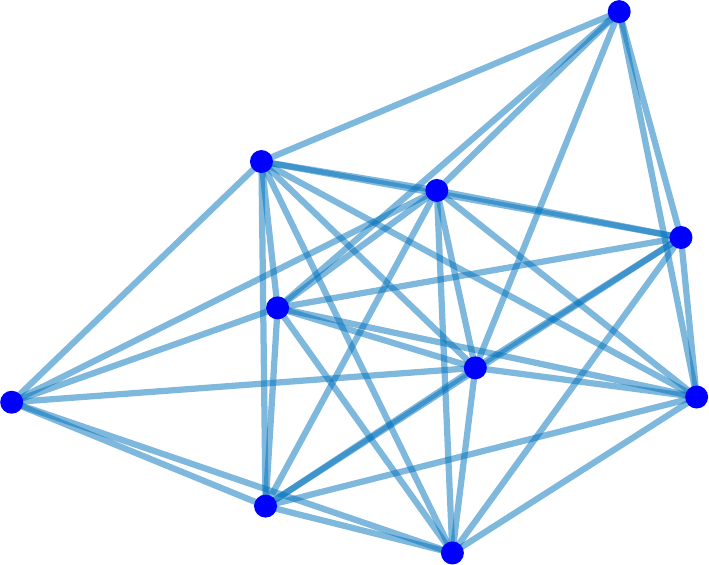}

$m=7$ \hskip 2.4truecm $m=8$ \hskip 2.6truecm $m=9$ \hskip 2.8truecm $m=10$

    \caption{\small Signable pseudo-invertible graphs on $3\le m \le10$ vertices with a maximal value of $\lambda_{max}$ (see Table~\ref{tab-spektrumpositivepseudo}).
    }
    \label{fig-graf-lambdamax-signablepseudoinvertible-even}
\end{figure}

In the next proposition, we derive upper and lower bounds of the spectral indices $\Lambda^{gap}, \Lambda^{ind}$, and $\Lambda^{pow}$ in the class of all signable pseudo-invertible graphs. 

\begin{proposition}
Assume $G^A$ is a signable pseudo-invertible graph of order $m$. Then
\begin{itemize}
    \item [i)]  $\Lambda^{gap}(G^A)\le 2\sqrt{\lfloor m/2\rfloor \lceil m/2\rceil}$, and $\Lambda^{ind}(G^A)\le \sqrt{\lfloor m/2\rfloor \lceil m/2\rceil}$. The equalities are attained by the complete bipartite graph $K_{m_1,m_2}$ where $m_1=\lceil m/2\rceil, m_2=\lfloor m/2\rfloor$.
    \item [ii)] $\Lambda^{pow}(G^A)\ge 2\sqrt{m-1}$. The equality is attained by the complete bipartite star graph $K_{m-1,1}\equiv S_m$.
\end{itemize}
\end{proposition}

\begin{proof}
The proof of part i) is based on the properties of the second largest eigenvalue $\lambda_2(A)$ of the graph $G^A$. With regard to Powers \cite{Pow, Powers1989}, and Cvetkovi\'c and Simi\'c \cite{Cvetkovic1995} we have the following estimate for the second largest eigenvalue: $-1\le \lambda_2(A)\le \lfloor m/2\rfloor -1$.   According to  Smith \cite{smith},   a simple connected graph $G^A$ has exactly one positive eigenvalue $\lambda_1(A)=\lambda_{max}>0$, i.e. $\lambda_2(A)\le 0$, if and only if it is a complete multipartitioned graph $K_{m_1, \dots, m_k}$ where $1\le m_1\le \dots \le m_k$ denote the sizes of the partitions, $m_1+ \dots + m_k =m$, and $k\ge2$ is the number of partitions (see also   Cvetkovi\'c \emph{et al.}  \cite[Theorem 6.7]{CvDS}). Therefore, for a graph $G^A$ different from any complete multipartitioned graph $K_{m_1, \dots, m_k}$ we have $\lambda_2(A)>0$, and consequently, 
$
- \sqrt{\lfloor m/2\rfloor \lceil m/2\rceil} \le \lambda_{min}(A)\le \lambda_-(A)<0<\lambda_+(A) \le  \lambda_2(A)\le \lfloor m/2\rfloor -1.
$
Hence the spectral gap $\Lambda^{gap}(G^A)=\lambda_+(A)-\lambda_-(A)\le
\sqrt{\lfloor m/2\rfloor \lceil m/2\rceil} + \lfloor m/2\rfloor -1\le 2\sqrt{\lfloor m/2\rfloor \lceil m/2\rceil}$ because $\lfloor m/2\rfloor -1\le \sqrt{\lfloor m/2\rfloor \lceil m/2\rceil}$ for any $m$. Similarly, $\Lambda^{ind}(A)\le \sqrt{\lfloor m/2\rfloor \lceil m/2\rceil}$. 

To prove ii), we recall that $\Lambda^{pow}(G^A) \ge 2\sqrt{m-1}$ for a general simple connected graph $G^A$ (cf. Caporossi \emph{et al.} \cite[Theorem 2]{Cap}). The minimal value of $\Lambda^{pow}(G^A)=2\sqrt{m-1}$ is attained by the star graph $S_m\equiv K_{m-1,1}$ which is a signable positively and negatively pseudo-invertible bipartite graph, and the proof follows.
\end{proof}

\begin{figure}[ht]
    \centering

    \includegraphics[height=.15\textwidth, angle=90]{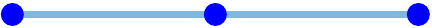}
    \ 
    \includegraphics[width=.2\textwidth]{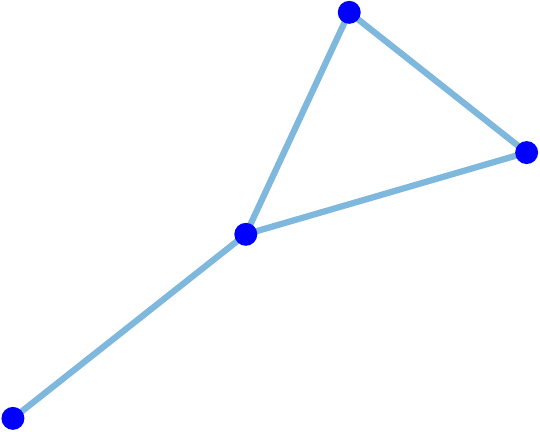}
    \ 
    \includegraphics[width=.2\textwidth]{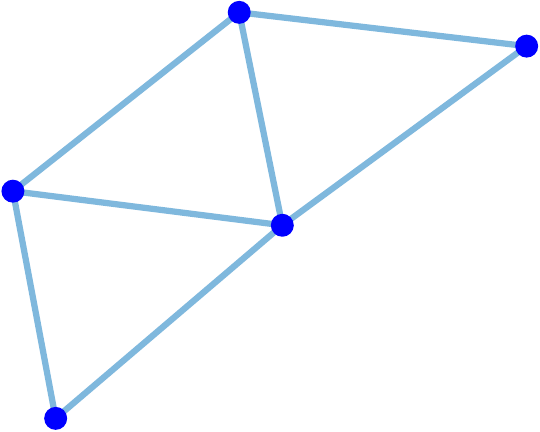}
    \ 
    \includegraphics[width=.2\textwidth]{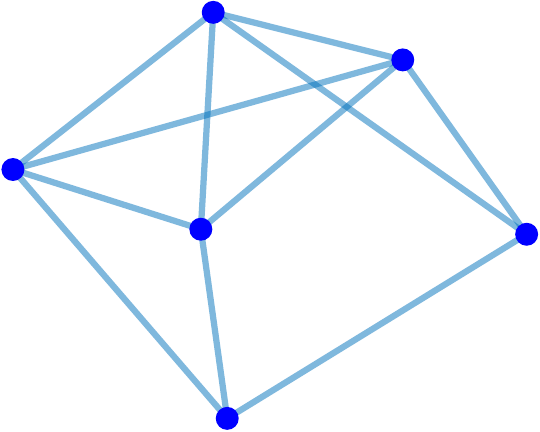}

$m=3$ \hskip 2.4truecm $m=4$ \hskip 2.6truecm $m=5$ \hskip 2.8truecm $m=6$

\vglue 0.4truecm

    \includegraphics[width=.2\textwidth]{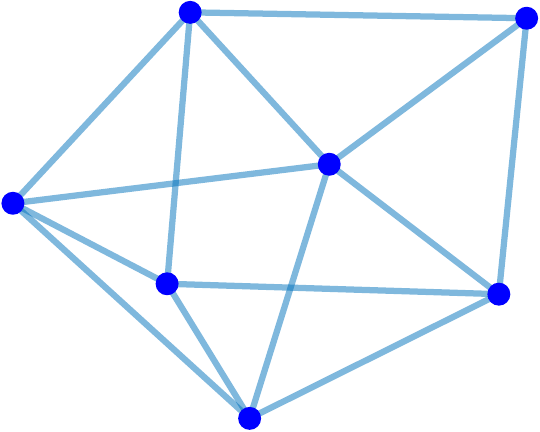}
    \ 
    \includegraphics[width=.2\textwidth]{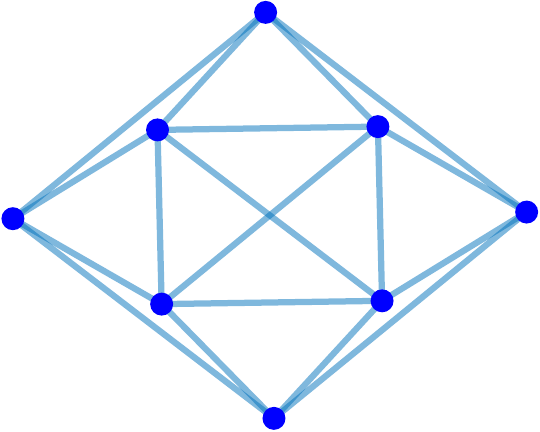}
    \ 
    \includegraphics[width=.2\textwidth]{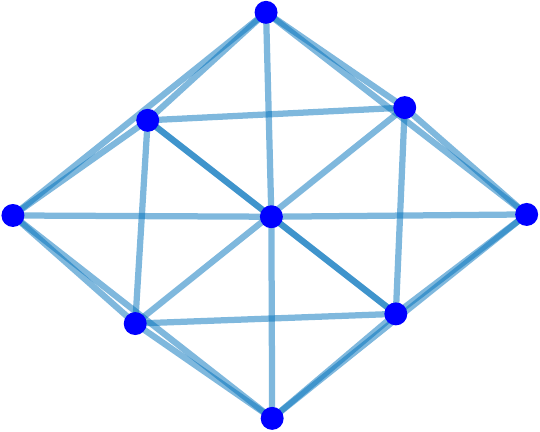}
    \ 
    \includegraphics[width=.25\textwidth]{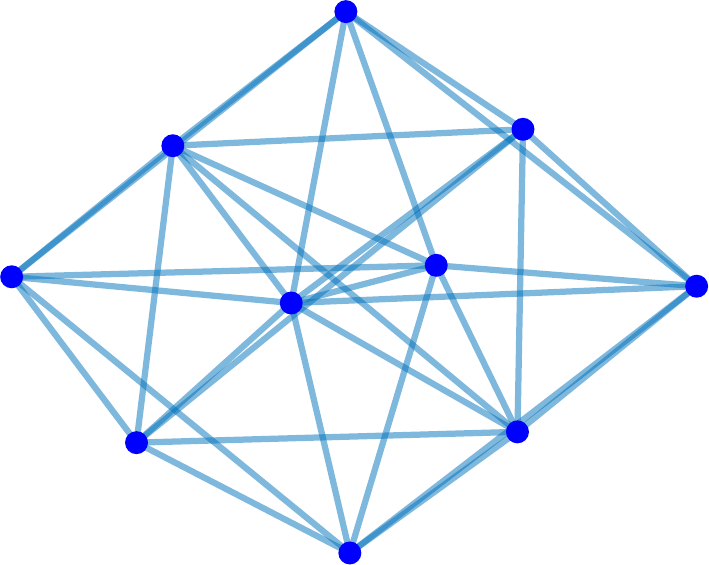}

$m=7$ \hskip 2.4truecm $m=8$ \hskip 2.6truecm $m=9$ \hskip 2.8truecm $m=10$

    \caption{\small Signable pseudo-invertible graphs on $3\le m\le 10$ vertices with a maximal value of $\Lambda^{pow}$ (see Table~\ref{tab-spektrumpositivepseudo}).
    }
    \label{fig-graf-lambdpoweramax-signablepseudoinvertible}
\end{figure}

In Fig.~\ref{fig-graf-lambdpoweramax-signablepseudoinvertible} we present signable pseudo-invertible graphs on the $4\le m \le 10$ vertices with maximal values of $\Lambda^{pow}$.
Graphs that achieve minimal values of $\Lambda^{gap}$ are shown in Fig.~\ref{fig-graf-lambdaHLgap-signablepseudoinvertible}. For $m$ even the minimal values of $\Lambda^{gap}$ and $\Lambda^{ind}$ are attained by signable integrally invertible graphs. 

\begin{figure}[ht]
    \centering
\bigskip
    \includegraphics[height=.15\textwidth, angle=90]{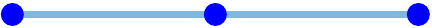 }
    \quad
    \includegraphics[width=.2\textwidth]{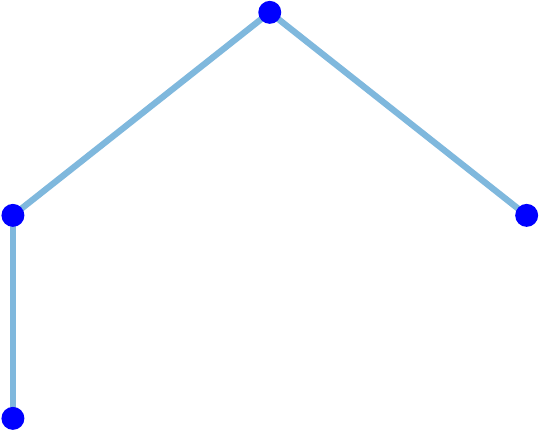 }
    \    
    \includegraphics[width=.2\textwidth]{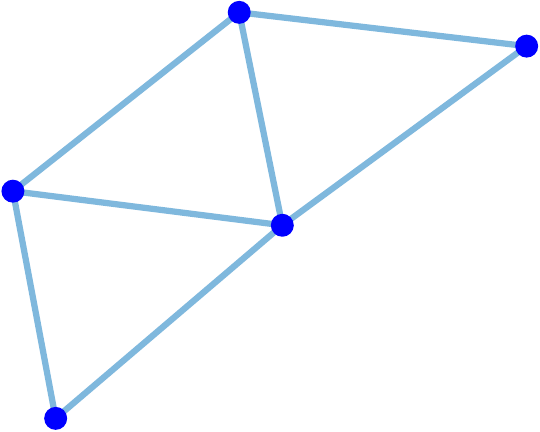 }
    \ 
    \includegraphics[width=.2\textwidth]{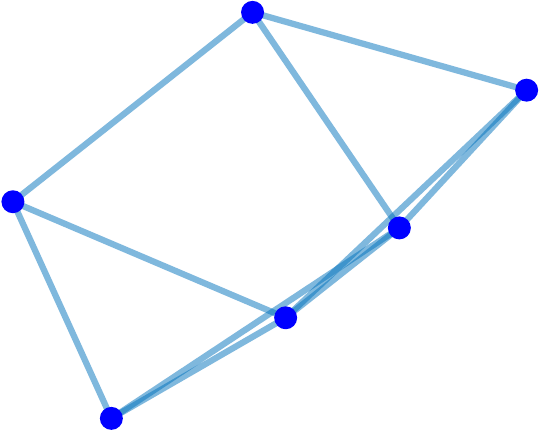 }

    $m=3$ \hskip 2.4truecm $m=4$ \hskip 2.6truecm $m=5$ \hskip 2.8truecm $m=6$

\medskip 
    \includegraphics[width=.2\textwidth]{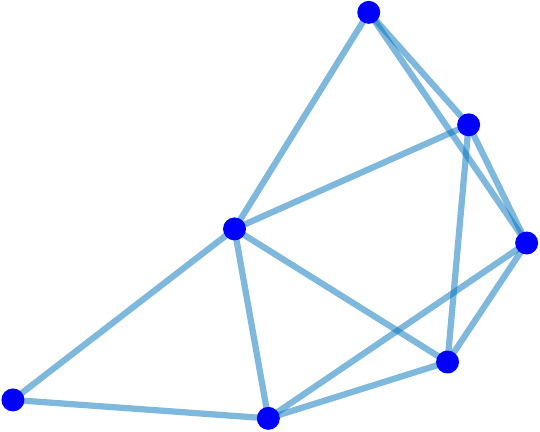 }
    \ 
    \includegraphics[width=.2\textwidth]{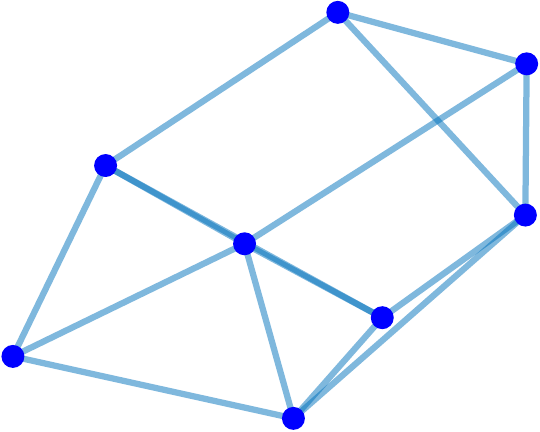 }
    \ 
    \includegraphics[width=.2\textwidth]{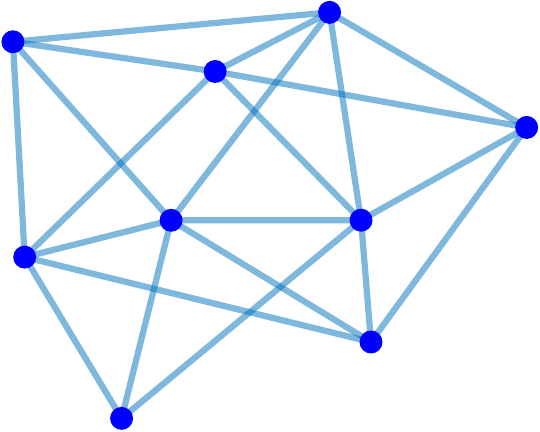}
    \ 
    \includegraphics[width=.26\textwidth]{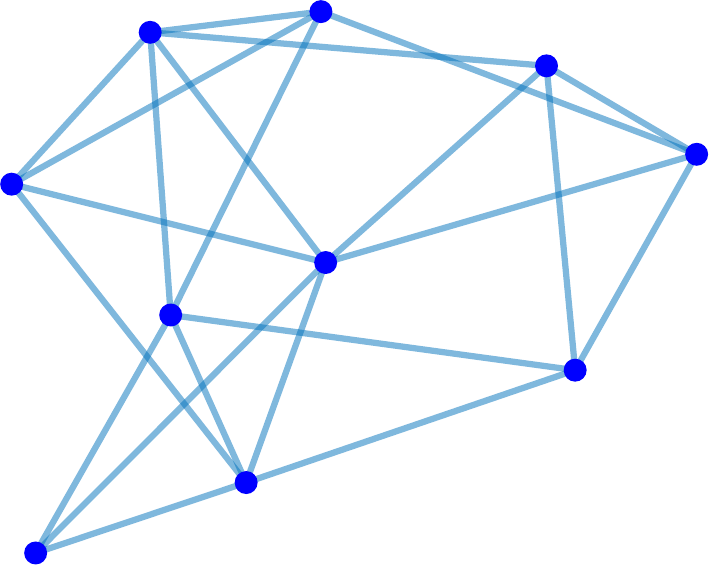}

$m=7$ \hskip 2.4truecm $m=8$ \hskip 2.6truecm $m=9$ \hskip 2.8truecm $m=10$

\caption{\small Signable pseudo-invertible graphs on $3\le m\le 10$ vertices with a minimal value of $\Lambda^{gap}$ (see Table~\ref{tab-spektrumpositivepseudo}).
    }
 \label{fig-graf-lambdaHLgap-signablepseudoinvertible}
\end{figure}

\subsection{All simple connected signable integrally invertible graphs}

In this section, we present signable integrally invertible graphs attaining extreme values of spectral indices. First, we present a simple construction of an infinite family of signable integrally invertible graphs. For arbitrary simple connected graph $G^B$ (not necessarily even pseudo-invertible) we can construct a signable integrally invertible graph just by adding a pendant vertex to each vertex of $G^B$. An adjacency matrix $A$ is integrally invertible ($|det(A)|=1$), iff $m$ is even. If $G_B$ is a graph, the graph $G_A$ formed from $G_B$ by adding a pendent vertex to each vertex of $G_B$ is known in the literature as the corona graph.

\begin{proposition}\label{prop-ii-pendant}
Assume $G^A$ is a graph of even order $m$ constructed from a simple connected $G^B$ of order $m/2$ by adding a pendant vertex to each vertex of $G^B$. Then
\begin{itemize}
    \item[i)] $G^A$ is a signable negatively integrally invertible graph.
    \item[ii)] If $G^B$ is a bipartite graph, then $G^A$ is a signable positively integrally invertible bipartite graph. 
    \item[iii)] $\Lambda^{pow}(G^A) = \sum_{\mu\in \sigma(B)} \sqrt{\mu^2 +4} \ge \max\{\Lambda^{pow}(G^B), m \}$.
    \item[iv)] The graph $G^A$ is homothetically pseudo-invertible.
\end{itemize}
\end{proposition}

\begin{proof}
The adjacency matrix $A$ of the graph $G^A$ and its inverse matrix $A^{-1}$ have the following block matrix form:
\begin{equation}
A = \left(
\begin{array}{cc}
0 & I\\
I & B
\end{array}
\right), \qquad\qquad 
A^{-1} = \left(
\begin{array}{cc}
-B & I\\
I & 0
\end{array}
\right), 
\label{blockmatrixGk}
\end{equation}
where $B$ is the adjacency matrix of the graph $G^B$. The matrix $A^{-1}$ is integral. Therefore, $A$ is integrally invertible, $|det(A)|=1$. Furthermore, it is negatively signable by the signature matrix $D=diag(I,-I)$, and statement i) follows. 

To prove ii), suppose that $G^B$ is a bipartite graph. Then there exists a $k\times k$ signature matrix $D_-$ such that $D_-B  D_- \le 0$. In fact, if we set $(D_-)_{ii} =1$ for a vertex $i$ belonging to the first bipartition of $G^B$, and $(D_-)_{jj} = -1$ for a vertex $j$ belonging to the second bipartition, then $D_-B D_- \le 0$. Since $D_- D_- = I$ we have $D A^{-1} D \ge 0$ where $D=diag(D_-, D_-)$ is a diagonal $m\times m$ signature matrix. Hence $G^A$ is also a positively integrally invertible bipartite graph. Statement ii) now follows. 
To prove iii), notice that $\lambda\in\sigma(A)$ iff $\lambda=(\mu\pm \sqrt{\mu^2 +4})/2$ for some $\mu\in\sigma(B)$. Therefore $\lambda\in\sigma_+(A)$ iff $\lambda=(\mu + \sqrt{\mu^2 +4})/2$ for some $\mu\in\sigma(B)$. As $\sum_{\lambda\in \sigma(A)} \lambda =0$, and $\sum_{\mu\in \sigma(B)} \mu =0$ we have
\[
\Lambda^{pow}(A) = \sum_{\lambda\in\sigma(A)} |\lambda| = \sum_{\lambda\in\sigma_+(A)} \lambda  - \sum_{\lambda\in\sigma_-(A)} \lambda = 2 \sum_{\lambda\in\sigma_+(A)}\lambda = \sum_{\mu\in \sigma(B)} \sqrt{\mu^2 +4} \ge \max\{\Lambda^{pow}(G^B), m \},
\]
because $\sqrt{\mu^2 +4} \ge  \max\{|\mu|, 2 \}$. The proof of iii) now follows. Finally, if $D$ is a signature matrix of the inverse graph $(G^A)^{-1}$ then $DA^{-1} D = \kappa P A P^T$ where 
$P= \left(
\begin{array}{cc}
0 & I\\
I & 0
\end{array}\right)$ is a permutation matrix and $\kappa=1$ or $\kappa=-1$. The proof of iv) now follows. 
\end{proof}

In Table~\ref{tab-spektrumall-intinv} we present descriptive statistics of the extreme eigenvalues $\lambda_{max}$, $\lambda_{min}$, spectral gap $\Lambda^{gap}$, spectral index $\Lambda^{ind}$, and spectral power $\Lambda^{pow}$ for all signable integrally invertible simple connected graphs on $m\le 10$ vertices.
Interestingly enough, the maximal and minimal eigenvalues $\lambda_{max}$ and $\lambda_{min}$, maximal gap $\Lambda^{gap}$ and the power $\Lambda^{pow}$ over all signable integrally invertible graphs on $m=6$ vertices are attained by the same graph (see Fig.~\ref{fig-graf-lambdamax-signableinvertible}).

\begin{table}[ht]
\caption{\small Descriptive statistics of the maximal(minimal) eigenvalues $\lambda_{max}$ ($\lambda_{min}$), spectral gap $\Lambda^{gap}$, spectral index $\Lambda^{ind}$, and spectral power $\Lambda^{pow}$ for all signable integrally invertible simple connected graphs on $m\le 10$ vertices.
}
\scriptsize
\begin{center}
\hglue -0.5truecm \begin{tabular}{l||l|l|l|l}
\hline
$m$  & $4$  & $6$  & $8$  & $10$\\
\hline\hline
$E(\lambda_{max})$   &
 1.8941 & 
 2.7716 & 
 3.6912 & 
 4.7082 
\\
$\sigma(\lambda_{max})$ &
 0.3904 & 
 0.5081 & 
 0.5868 & 
 0.6155 
\\
${\mathcal S}(\lambda_{max})$  &
 0 & 
 0.0654 & 
 -0.0866 & 
 -0.1249 
\\
${\mathcal K}(\lambda_{max})$ &
 1 & 
 2.2008 & 
 2.6776 & 
 2.8808 
 \\
$\max(\lambda_{max})$ &
 2.1701 $(+)$& 
 3.7321 $(+)$& 
 5.3628 $(+)$& 
 7.1205 $(+)$
\\
$\min(\lambda_{max})$  &
 1.6180 $(\pm)$& 
 1.8019 $(\pm)$& 
 1.8794 $(\pm)$& 
 1.9190 $(\pm)$
\\
\hline
$E(\lambda_{min})$ &
 -1.5496 & 
 -1.9924 & 
 -2.4068 & 
 -2.7995  
\\
$\sigma(\lambda_{min})$ &
 0.0967 & 
 0.2384 & 
 0.2555 & 
 0.2680 
\\
${\mathcal S}(\lambda_{min})$ &
 0 & 
 -0.1743 & 
 -0.1380 & 
 -0.1912 
\\
${\mathcal K}(\lambda_{min})$ &
 1 & 
 1.8261 & 
 2.8754 & 
 2.9341 
\\
$\max(\lambda_{min})$ &
 -1.4812 $(+)$& 
 -1.6180 $(+)$& 
 -1.6180 $(+)$& 
 -1.6180 $(+)$ 
\\
$\min(\lambda_{min})$&
 -1.6180 $(\pm)$& 
 -2.4142 $(+)$& 
 -3.3028 $(+)$& 
 -4.2361 $(+)$
\\
\hline
$\max(\Lambda^{gap})$ &
 1.3111 $(+)$& 
 1.2679 $(+)$& 
 1.2947 $(+)$& 
 1.4773 $(+)$
\\
$\max(\Lambda^{ind})$ &
 1.0000 $(+)$ & 
 1.0000 $(+)$& 
 1.2501 $(+)$& 
 1.4685 $(+)$
\\
$\max(\Lambda^{pow})$ &
 4.9624 $(+)$& 
 8.8284 $(+)$& 
 13.4838 $(+)$& 
 18.8636 $(+)$
\\
\hline
$\min(\Lambda^{gap})$ &
 1.2360 $(\pm)$& 
 0.7423 $(+)$& 
 0.3877 $(+)$& 
 0.1647 $(+)$
\\
$\min(\Lambda^{ind})$ &
 0.6180 $(\pm)$& 
 0.4142 $(\pm)$& 
 0.2624 $(+)$& 
 0.1092 $(+)$
\\
$\min(\Lambda^{pow})$ &
 4.4720 $(\pm)$& 
 6.8990 $(\pm)$& 
 9.2916 $(\pm)$& 
 11.6568 $(\pm)$
\\
\hline

\end{tabular}
\end{center}
\label{tab-spektrumall-intinv}
\end{table}

\begin{figure}[ht]
    \centering

    \includegraphics[width=.2\textwidth]{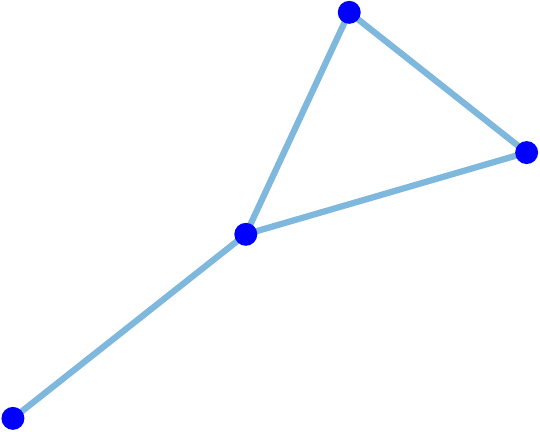}
    \ 
    \includegraphics[width=.2\textwidth]{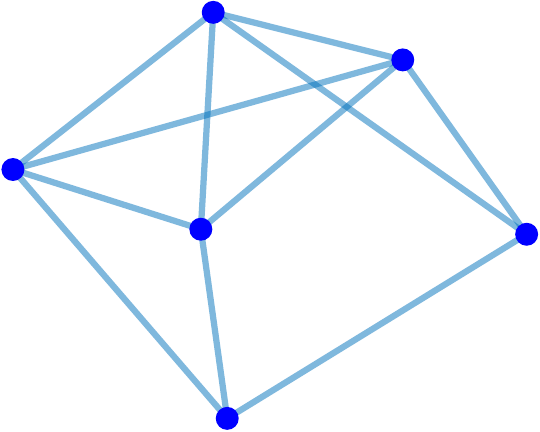}
    \ 
    \includegraphics[width=.23\textwidth]{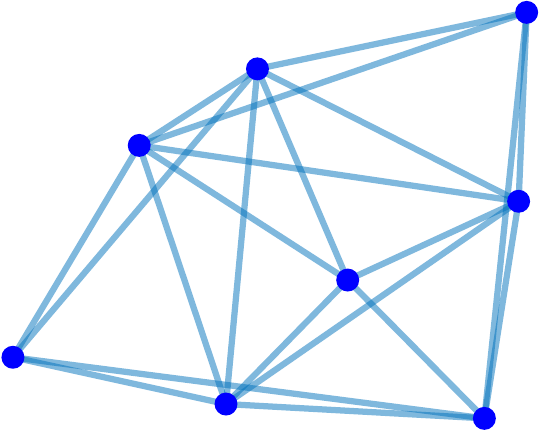}
    \ 
    \includegraphics[width=.3\textwidth]{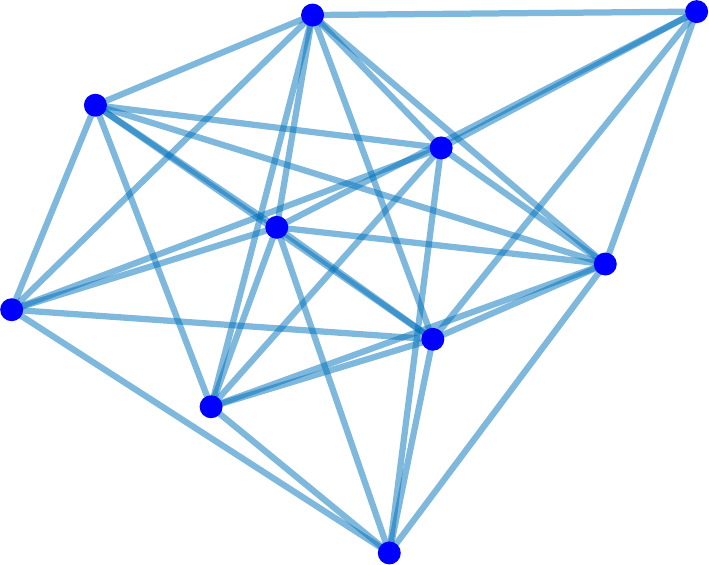}

\vglue 0.4truecm

    \includegraphics[width=.2\textwidth]{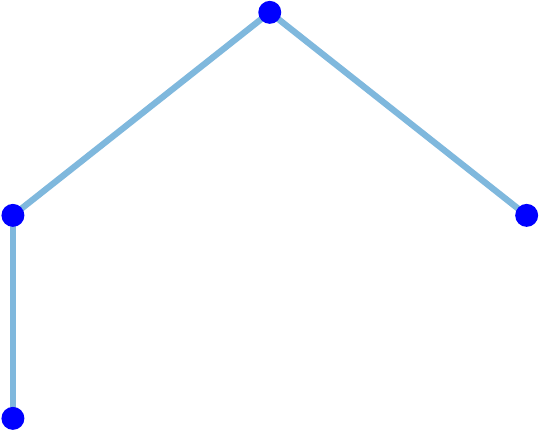}
    \ 
    \includegraphics[width=.2\textwidth]{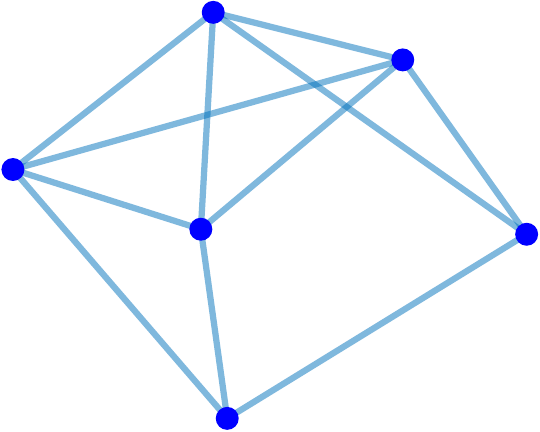}
    \ 
    \includegraphics[width=.23\textwidth]{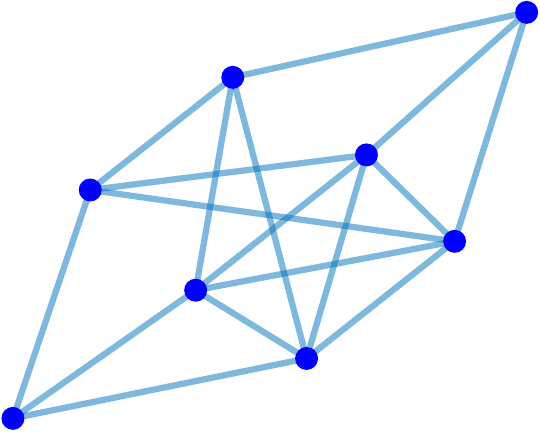}
    \ 
    \includegraphics[width=.3\textwidth]{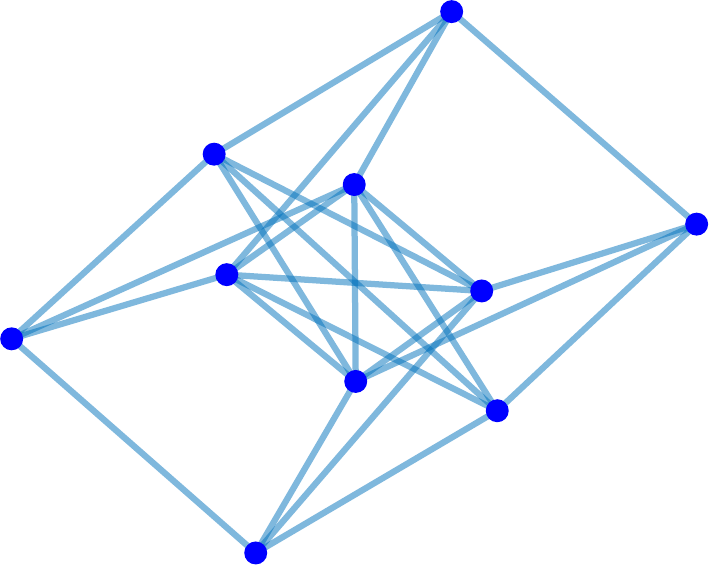}

$m=4$ \hskip 2.4truecm $m=6$ \hskip 2.6truecm $m=8$ \hskip 2.8truecm $m=10$

    \caption{\small Signable integrally invertible graphs on the $m=4,6,8,10$ vertices with a maximal value of $\lambda_{max}$ (top row) and a minimal value of $\lambda_{min}$ (bottom row) (see Table~\ref{tab-spektrumall-intinv}).
    }
    \label{fig-graf-lambdamax-signableinvertible}
\end{figure}

\begin{remark}\label{extremesignability}
In the class of all pseudo-invertible graphs, the maximal and minimal values of spectral indices $\lambda_{max},\lambda_{min}, \Lambda^{gap}, \Lambda^{ind}, \Lambda^{pow}$, can be attained by graphs with varying signability (see Table~\ref{tab-spektrumpositivepseudo}). For example, there are four ($m=5,8,9,10$) negatively but only three ($m=4,6,7$) positively pseudo-invertible graphs with a maximal value of $\lambda_{max}$ for the order $4\le m\le 10$. Similarly, maximal values of $\Lambda^{pow}$ are attained by negative pseudo-invertible graphs for order $m=5,7,9,10$. 
On the other hand, it follows from the results summarized in Table~\ref{tab-spektrumall-intinv} that the extreme values of these spectral indices are achieved only by integrally invertible positive ($+$), or positively/negatively ($\pm$) integrally invertible graphs.

\end{remark}

Finally, we present signable integrally invertible graphs of the order $m=4,6,8,10$ achieving extreme spectral indices. In Fig.~\ref{fig-graf-lambdamax-signableinvertible} we plot signable integrally invertible graphs with a maximal value of $\lambda_{max}$, and minimal value of $\lambda_{min}$. In Fig.~\ref{fig-graf-lambdpoweramax-signableintegrallyinvertible}  we show signable integrally invertible graphs with maximal and minimal values of $\Lambda^{pow}$. The minimal value of $\Lambda^{pow}=m -4 +2 \sqrt{m/2 + 3}$ in the class of all signable integrally invertible graphs is attained by a graph constructed from the star graph $S_{m/2}$  by adding one pendant vertex to each vertex of $S_{m/2}$ (see Fig.~\ref{fig-graf-lambdpoweramax-signableintegrallyinvertible} and Proposition~\ref{prop-ii-pendant} ). In Fig.~\ref{fig-graf-lambdaHLgap-signableintegrallyinvertible} we depict signable integrally invertible graphs with the maximal value of $\Lambda^{gap}$.

\begin{figure}[ht]
    \centering

    \includegraphics[width=.2\textwidth]{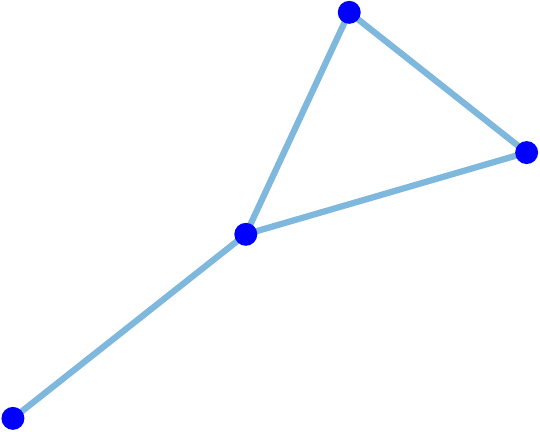 }
    \ 
    \includegraphics[width=.2\textwidth]{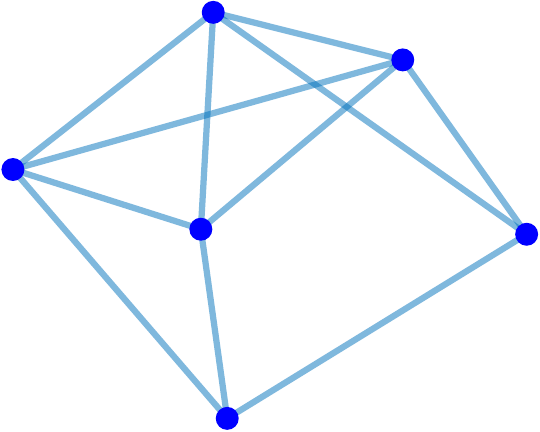 }
    \ 
    \includegraphics[width=.2\textwidth]{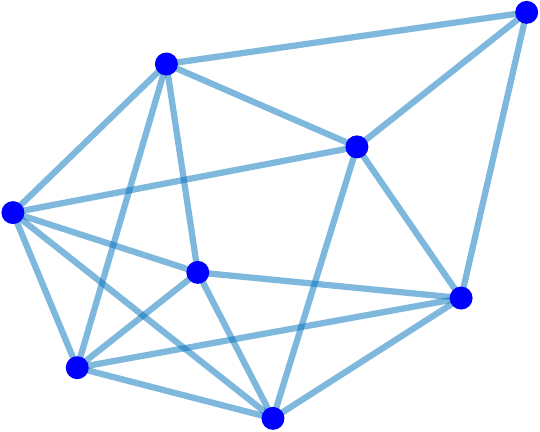 }
    \ 
    \includegraphics[width=.26\textwidth]{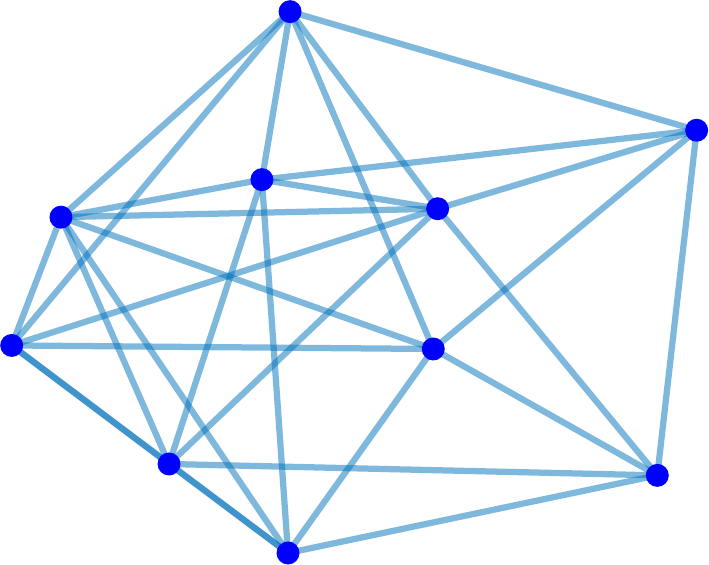}

\bigskip

    \includegraphics[width=.2\textwidth]{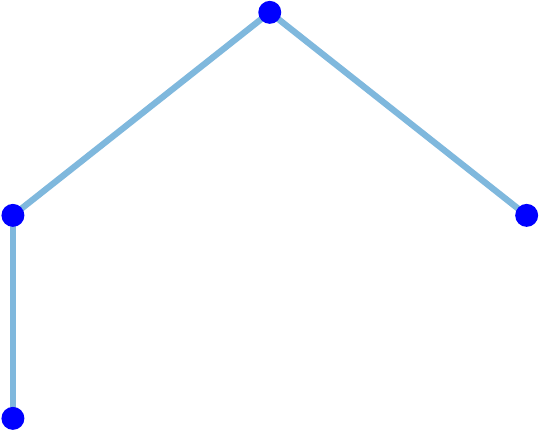 }
    \ 
    \includegraphics[width=.2\textwidth]{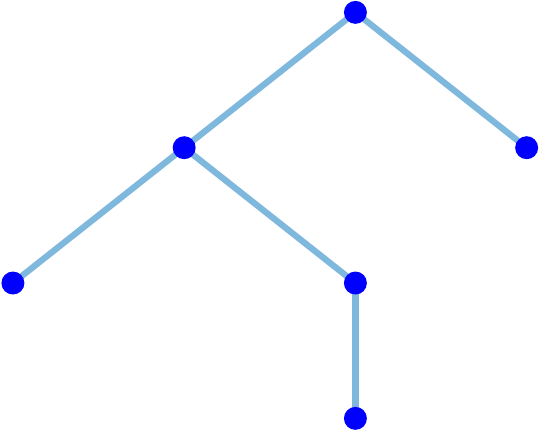 }
    \ 
    \includegraphics[width=.2\textwidth]{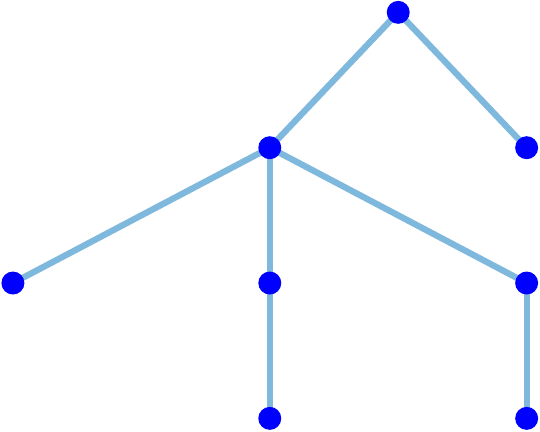 }
    \ 
    \includegraphics[width=.26\textwidth]{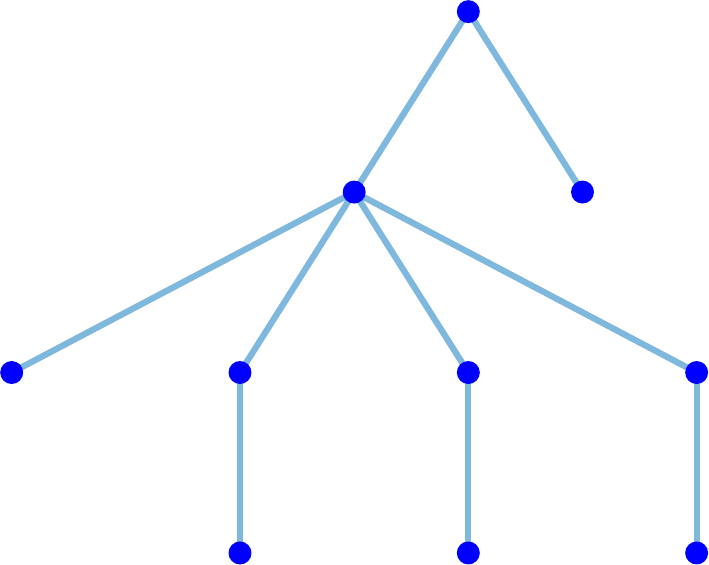}

$m=4$ \hskip 2.4truecm $m=6$ \hskip 2.6truecm $m=8$ \hskip 2.8truecm $m=10$

    \caption{\small Signable integrally invertible graphs on $m=4,6,8,10$ vertices with a maximal (top row) and minimal (bottom row) value of $\Lambda^{pow}$ (see Table~\ref{tab-spektrumall-intinv}).
    }
    \label{fig-graf-lambdpoweramax-signableintegrallyinvertible}
\end{figure}

\begin{figure}[ht]
    \centering

    \includegraphics[width=.2\textwidth]{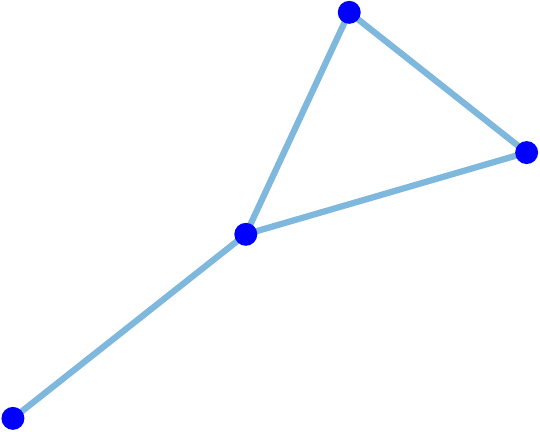 }
    \ 
    \includegraphics[width=.2\textwidth]{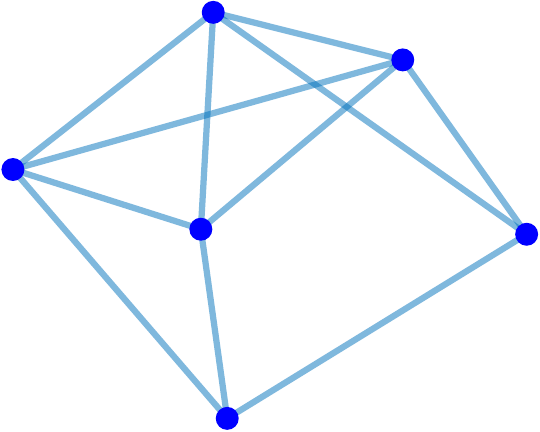 }
    \ 
    \includegraphics[width=.2\textwidth]{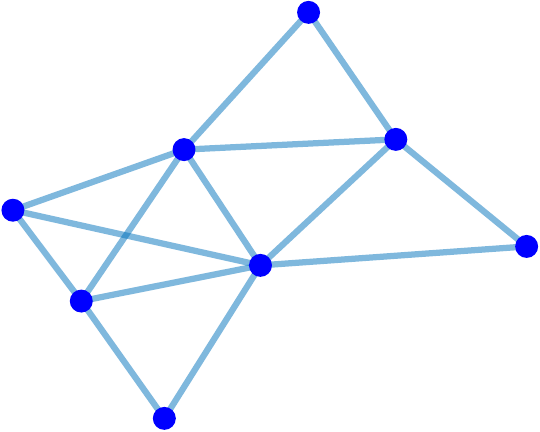 }
    \ 
    \includegraphics[width=.26\textwidth]{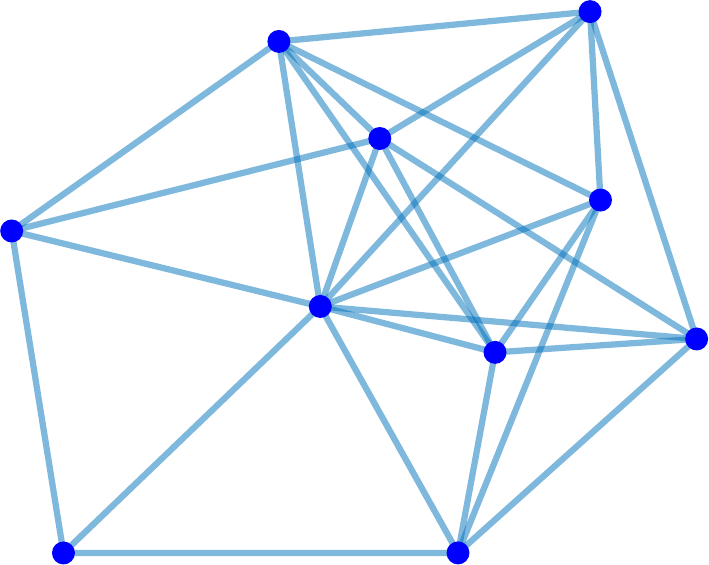}

$m=4$ \hskip 2.4truecm $m=6$ \hskip 2.6truecm $m=8$ \hskip 2.8truecm $m=10$

    \caption{\small Signable integrally invertible graphs on $m=4,6,8,10$ vertices with a maximal value of $\Lambda^{gap}$ (seeTable~\ref{tab-spektrumall-intinv}).
    }
    \label{fig-graf-lambdaHLgap-signableintegrallyinvertible}
\end{figure}

\section{Conclusions}
In this paper, we investigated signable pseudo-invertible graphs. We analyzed various qualitative, statistical, and extreme properties of spectral indices of pseudo-invertible graphs. We derived various upper and lower bounds on maximal and minimal eigenvalues and spectral indices $\Lambda^{gap}, \Lambda^{ind}$, and $\Lambda^{pow}$. We showed that the complete multipartitioned graphs $K_{m_1,\dots,m_k}$ are signable pseudo-invertible only for $k=2$. A novel concept of $G^\mathscr{A}$-complete multipartitioned graph was introduced and analyzed. It is a natural generalization of a complete multipartitioned graph. We also presented computational and statistical results for the class of signable pseudo-invertible and integrally invertible graphs of order $m\le 10$. Some properties of extreme eigenvalues and indices were shown for arbitrary order $m$.

\end{document}